\documentclass[12pt,epsfig,amsfonts]{amsart} 
\usepackage{amsmath,amsthm,amssymb,amscd,epsfig,color}
\usepackage{graphicx}
\usepackage{mathrsfs}

\newcommand{\1}{\mbox{1}\hspace{-0.25em}\mbox{l}}

\newtheorem*{theorema}{Theorem A}
\newtheorem*{theoremb}{Theorem B}
\newtheorem*{theoremc}{Theorem C}

\newtheorem{prop}{Proposition}[section]
\newtheorem{lemma}[prop]{Lemma}
\newtheorem{remark}[prop]{Remark}

\newtheorem{corollary}[prop]{Corollary}

\theoremstyle{definition}
\newtheorem{definition}[prop]{Definition}

\numberwithin{equation}{section}

\usepackage[utf8]{inputenc}
\usepackage{fancyvrb}

\begin{document}

\author{ Hiroki Takahasi and Shintaro Suzuki}

\address{Keio Institute of Pure and Applied Sciences (KiPAS), Department of Mathematics,
Keio University, Yokohama,
223-8522, JAPAN} 
\email{hiroki@math.keio.ac.jp}
\email{shin-suzuki@math.keio.ac.jp}


\title[Distribution of
cycles for random dynamical systems]{Distribution of 
cycles for\\one-dimensional random dynamical systems}

\begin{abstract}
We consider an independently identically distributed random dynamical system generated by finitely many, non-uniformly expanding Markov interval maps with a finite number of branches.
Assuming a topologically mixing condition and the uniqueness of equilibrium state for the associated skew product map, we 
establish a samplewise (quenched) almost-sure level-2 weighted equidistribution of
``random cycles'', with respect to a natural 
stationary measure as the periods of the cycles tend to infinity. This result implies an analogue of Bowen's theorem on periodic orbits of topologically mixing Axiom~A diffeomorphisms. 
We also prove another almost-sure convergence theorem, as well as an averaged (annealed) theorem that is related to semigroup actions. 
We apply our results to the random $\beta$-expansion of real numbers,
and obtain almost-sure convergences of average digital quantities in random $\beta$-expansions of random cycles that do not follow from 
the 
application of the ergodic theorems of Birkhoff or Kakutani.
Our main results are applicable to random dynamical systems generated by finitely many maps with common neutral fixed points.
\end{abstract}
\maketitle

\tableofcontents

\section{Introduction}
One leading idea in the qualitative understanding of deterministic dynamical systems is 
to use collections of periodic orbits 
to structure the dynamics. This idea traces back to 
Poincar\'e \cite{Poi92}, and has been supported
by Bowen  \cite{Bow71,Bow74} who proved that periodic orbits 
   of topologically mixing Axiom~A diffeomorphisms equidistribute with respect
   to the measure of maximal entropy.
   The importance of periodic orbits in descriptions of ergodic properties of natural
   invariant probability measures has long been recognized in the physics literature, see e.g., \cite{Cvi88,GOY88}.

Deterministic dynamical systems are iterations of the same map, whereas
random dynamical systems
are compositions of different maps chosen at random.
Therefore, for the latter
it is not apparent how periodic orbits should be defined, or what should play the role of periodic orbits.
For random subshifts of finite type, Kifer constructed 
a certain substitute for periodic orbits
\cite[Appendix]{Kif00}. In there, 
he raised a conjecture on a random Livschitz theorem, with a view that 
``periodic orbits'' in random setup should play an important role too, as in deterministic dynamical systems.

 This paper attempts to shed some light in the  direction that pursues the importance of ``periodic orbits'' in random setup. 
       We are concerned with an independently identically distributed (i.i.d.) random dynamical system generated by finitely many piecewise differentiable maps $T_1,\ldots,T_N$ $(N\geq2)$ of a compact interval $X$ in which  the map $T_i$ is chosen with positive probability $p_i$ at each step.
More precisely, we are concerned with a probability space $(\Omega,m_p)$ that is the infinite product of $(\{1,\ldots,N\},p)$,
where $\Omega=\{1,\ldots,N\}^{\mathbb N}$ is the sample space and 
 $m_p$ is the Bernoulli measure 
determined by the $N$-dimensional positive probability vector
 $p=(p_1,\ldots,p_N)$.
For each sample 
$\omega\in\Omega$ 
we consider a random composition \[T_\omega^n=T_{\omega_n}\circ T_{\omega_{n-1}}\circ \cdots\circ T_{\omega_1}\quad (n=1,2,\ldots),\] and write
$T_\omega^0$ for the identity map on $X$.
Note that $T_\omega^n$ depends only on the first $n$ symbols of $\omega$.
 Put ${\rm Fix}(T_\omega^n)=\{x\in X\colon T_\omega^n(x)=x\}.$
By a {\it random cycle}
we mean an element of the set \[\bigcup_{\omega\in\Omega}\bigcup_{n=1}^\infty{\rm Fix}(T_\omega^n).\]

In the special case $T_1=T_2=\cdots= T_N$,
random cycles are nothing but periodic points of 
the deterministic dynamical system generated by the iteration of $T_1$.
In general, 
random cycles in ${\rm Fix}(T_\omega^n)$ are considered to be natural substitutes for periodic points of period $n$ for deterministic dynamical systems. 
Random cycles were used in \cite{Buz02,Rue90} for defining dynamical zeta functions in random setup.
A natural question is whether random cycles really carry relevant information of the dynamics.
A negative result is due to Buzzi \cite{Buz02}, who showed that a dynamical zeta function defined with random cycles of certain random matrices cannot be extended beyond its disk of holomorphy, almost surely.
 Buzzi's result might imply that 
random cycles were not so important.
All our results in this paper 
support the importance of random cycles.

\subsection{Statements of main results}

 A {\it Markov map on $X$} of class $C^{1+\tau}$, $0<\tau\leq1$
 is a map $T\colon \bigcup_{a\in \mathcal A(0)} J(a)\to X$ where
   $\mathcal A(0)$ is a finite subset of the set $\mathbb N=\{1,2,\ldots\}$ of positive integers, 
   and 
   $(J(a))_{a\in \mathcal A(0)}$ 
   is a partition of $X$ into pairwise disjoint subintervals
   such that: 
   
   \begin{itemize}
   \item[$\circ$] for each $a\in \mathcal A(0)$, 
    $T|_{J(a)}$ extends to
a $C^{1+\tau}$ diffeomorphism on ${\rm cl}(J(a))$;

\item[$\circ$] if $a,b\in\mathcal A(0)$ and $T(J(a))\cap{\rm int}(J(b))\neq\emptyset$, then ${\rm cl}(T(J(a)))\supset J(b)$,
\end{itemize}
where ${\rm int}(\cdot)$ and ${\rm cl}(\cdot)$ denote the interior and closure operations respectively.
 We call $(J(a))_{a\in \mathcal A(0)}$ a {\it Markov partition} for $T$. 
 The derivatives of $T$ at the boundaries of the elements of its Markov partition are the appropriate one-sided derivatives: $T'(x)=(T|_{J(a)})'(x)$ for $x\in J(a)\cap\partial J(a)$ and the same for higher order ones.
   We say $T$ is
    {\it non-uniformly expanding} if $|T'(x)|>1$ holds all $x\in X$ but at most finitely many points.
  If $\inf_{x\in X}|T'(x)|>1$, we say $T$ is
    {\it uniformly expanding}. 
    
    A deterministic dynamical system generated by
    iterations of a single uniformly expanding Markov map $T$ is an archetypal 
    model of chaotic dynamical systems that is qualitatively describable 
    through the thermodynamic formalism \cite{Bow75,Rue04,Sin72}. 
    The strategy is to ``code'' the system 
    by
    following its orbits over the Markov partition. This defines a one-dimensional spin system with an exponentially decaying interaction, and one can construct dynamically relevant invariant measures and study their properties borrowing ideas from equilibrium statistical mechanics.
    On distributions of periodic points,
    it follows from \cite{Bow71,Bow74} that
    if $T$ is topologically mixing, then 
\begin{equation}\label{bowen}\lim_{n\to\infty}\frac{1}{Z_n}\sum_{x\in{\rm Fix}(T^n)}|(T^n)'x|^{-1}\frac{1}{n}\sum_{k=0}^{n-1}\varphi(T^k(x))=\int \varphi d\lambda,\end{equation} 
    for any continuous function $\varphi\colon X\to\mathbb R$,
    where $T^n$ denotes the $n$ iteration of $T$, 
    $(T^n)'x=\prod_{k=0}^{n-1}T'(T^k(x))$, $Z_n>0$ the normalizing constant,
and $\lambda$ the unique invariant Borel probability measure that is absolutely continuous with respect to the Lebesgue measure. 
 Examples of non-uniformly expanding Markov maps in our mind which are not uniformly expanding are those with neutral fixed points, which naturally arise in number theory, or are considered as a 
 simple model of intermittency,
 see \cite{FPT89,I89,LSV99,PomMan80,Sch75,Tha83} for example.

For a topological space $\mathcal X$, let $\mathcal M(\mathcal X)$
denote the space of Borel probability measures on $\mathcal X$ endowed with the weak* topology. 
For $x\in \mathcal X$ let $\delta_x\in\mathcal M(\mathcal X)$ denote the unit point mass at $x$.
Our first main result is stated as follows, under the assumptions
formulated in Section~\ref{markov}
with $\lambda_p$ being the stationary measure there, see Definitions~\ref{def1} and \ref{def2}.
For $x\in X$, $\omega\in\Omega$ and $n\geq1$
 let
\[\delta_x^{\omega,n}=\frac{1}{n}\sum_{k=0}^{n-1}\delta_{T_\omega^k(x)},\]
and write $(T_\omega^n)'x=\prod_{k=0}^{n-1}T_{\omega_{k+1}}'(T_\omega^k(x)).$


\begin{theorema}
[Almost-sure level-2 weighted equidistribution of random cycles I]
Let $T_1,\ldots,T_N$ be non-uniformly expanding Markov maps on $X$ generating a nice, topologically mixing skew product Markov map. 
If the uniqueness of equilibrium state holds
 for an $N$-dimensional probability vector $p$, then
for $m_p$-almost every sample $\omega\in\Omega$
and any continuous function $\tilde \varphi\colon\mathcal M(X)\to\mathbb R$ 
we have
\[\lim_{n\to\infty}\frac{1}{Z_{\omega,n}}\sum_{x\in{\rm Fix}(T_\omega^n)}|(T^n_\omega)'x|^{-1}
\tilde \varphi(\delta_x^{\omega,n})=\tilde \varphi(\lambda_p),\]
where $Z_{\omega,n}>0$ denotes the normalizing constant.
\end{theorema}

 For each $\omega\in\Omega$, define a Borel probability measure $\tilde{\xi}^\omega_n$ on
 $\mathcal M(X)$ by
\begin{equation}\label{random-c}\tilde{\xi}^\omega_n=\frac{1}{Z_{\omega,n}}\sum_{x\in{\rm Fix}(T^n_\omega)}|(T_{\omega}^n)'x|^{-1}\delta_{\delta_{x}^{\omega,n}}\quad(n=1,2,\ldots).\end{equation}
We also define a Borel probability measure $\xi^\omega_n$ 
  on $X$ by
  \begin{equation}\label{measure-xi}{\xi}^\omega_n=\frac{1}{Z_{\omega,n}}\sum_{x\in{\rm Fix}(T^n_\omega)}|(T_{\omega}^n)'x|^{-1}\delta_{x}^{\omega,n}\quad(n=1,2,\ldots).\end{equation}
We have referred to Theorem~A as ``level-2'', since
the convergence there is equivalent to the convergence of
$(\tilde{\xi}^\omega_n)_{n=1}^\infty$
in the weak* topology 
to the unit point mass at $\lambda_p$ as $n\to\infty$.
From Theorem~A, we obtain 
the convergence of
$(\xi^\omega_n)_{n=1}^\infty$
in the weak* topology 
to $\lambda_p$ as $n\to\infty$, namely
\begin{equation}\label{bow-random}\lim_{n\to\infty}\frac{1}{Z_{\omega,n}}\sum_{x\in{\rm Fix}(T^n_\omega)}|(T_{\omega}^n)'x|^{-1}\frac{1}{n}\sum_{k=0}^{n-1}\varphi(T_\omega^k(x))=\int \varphi d\lambda_p,\end{equation}
for any continuous function $\varphi\colon X\to\mathbb R$. In other words,
 weighted random cycles equidistribute with respect to $\lambda_p$ almost surely.
Compare \eqref{bowen} and \eqref{bow-random}.
By the Portmanteau theorem, for any Borel subset $A$ of $X$ satisfying $\lambda_p(\partial A)=0$
we have 
\[\label{recover}\lim_{n\to\infty}\frac{1}{Z_{\omega,n}}\sum_{x\in{\rm Fix}(T^n_\omega)}|(T_{\omega}^n)'x|^{-1}\frac{1}{n}\#\{0\leq k\leq n-1\colon T_\omega^k(x)\in A\}=\lambda_p(A).\]
This equation is essentially a representation of the stationary measure $\lambda_p$
in terms of random cycles.

Theorem~A is a samplewise result, and it is natural to ask what are averaged behaviors over samples.
The averaging by integration 
yields one convergence result, see Corollary~\ref{thmb}.
Here we take a more intuitive way of sample averaging that ties in with semigroup actions in a particular case. Under the notation in Theorem~A, for $n\geq1$ and $\omega_1\cdots\omega_n\in \{1,\ldots,N\}^n$ we set
\[Q_{p}(\omega_1\cdots\omega_n)=\prod_{i=1}^Np_i^{\#\{1\leq k\leq n\colon \omega_k=i\}}.\]
\begin{theoremb}
[Averaged level-2 weighted equidistribution of random cycles I]
Let $T_1,\ldots,T_N$ and $p=(p_1,\ldots,p_N)$ be as in Theorem~A.
For
any continuous function $\tilde \varphi\colon\mathcal M(X)\to\mathbb R$ 
we have
\[\lim_{n\to\infty}\frac{1}{Z_{p,n}}\sum_{\omega_1\cdots \omega_n\in\{1,\ldots,N\}^n}Q_{p}(\omega_1\cdots\omega_n)\sum_{x\in{\rm Fix}(T_{\omega}^n)}|(T_{\omega}^n)'x|^{-1}\tilde \varphi(\delta_x^{\omega,n})=\tilde \varphi(\lambda_p),\]
where $Z_{p,n}>0$ denotes the normalizing constant.
\end{theoremb}
One known result relevant to Theorem~B is due to Carvalho, Rodrigues and Varandas \cite[Theorem~C]{CRV17}, in which they considered  semigroup actions generated by finitely many Ruelle-expanding maps, and showed that fixed points of semigroup elements of word length $n$ equidistribute with respect to the measure of maximal entropy as $n\to\infty$.
The random composition of $T_1,\ldots,T_N$ may be viewed as an action of a semigroup with $N$ generators: random cycles of period $n$ correspond to fixed points of semigroup elements of word length $n$.
In the equiprobability 
case $p_i=1/N$ for all $1\leq i\leq N$, the factors
$Q_{p}(\omega_1\cdots\omega_n)$ in the equation in Theorem~B are all equal to $1/N^n$ and they cancel out.
As a result, we obtain the following corollary.
\begin{corollary}[Level-2 weighted equidistribution of fixed points of semigroup actions]
Let $T_1,\ldots,T_N$ be as in Theorem~A. If the uniqueness of equilibrium state holds for the probability vector $p=(p_1,\ldots,p_N)$ with
$p_i=1/N$ for all $1\leq i\leq N$, then
for any continuous function $\tilde \varphi\colon\mathcal M(X)\to\mathbb R$ we have
\[\lim_{n\to\infty}\left(\sum_{\omega_1\cdots \omega_n\in \{1,\ldots,N\}^n}Z_{\omega,n}\right)^{-1}\sum_{\omega_1\cdots \omega_n\in\{1,\ldots,N\}^n}\sum_{x\in{\rm Fix}(T_{\omega}^n)}|(T_{\omega}^n)'x|^{-1}\tilde \varphi(\delta_x^{\omega,n})=\tilde \varphi(\lambda_p).\]
\end{corollary}

The method of proofs of our main results
 is a combination of the thermodynamic formalism and level-2
 large deviations 
 for deterministic dynamical systems.
Large deviation principles for random dynamical systems have already been formulated in \cite{Kif90,Kif08}. However, these results are concerned with occupational measures, 
while we deal with measures associated with random cycles.
Moreover, large deviations lower bounds are not necessary in proving 
our main results. 
 For a general account on large deviations, 
 including the precise meanings of the terms level-2 and level-1
 we refer the reader to the book of Ellis \cite[Chapter~1]{Ell85}.

Our random dynamical system is represented by a skew product map 
\begin{equation}\label{R-def}R\colon(\omega,x)\in\Lambda\mapsto(\theta\omega,T_{\omega_1}(x))\in\Lambda,\end{equation}
where $\Lambda=\Omega\times X$ and $\theta\colon\Omega\to\Omega$ denotes the left shift $(\theta \omega)_n=\omega_{n+1}$.
 Note that
$R^n(\omega,x)=(\theta^n\omega,T_{\omega}^n(x))$ for $n\geq0$.
The evolution of the second coordinate is of interest.
A key observation is that $x\in {\rm Fix}(T_\omega^n)$ 
implies
$R^n(\omega',x)=(\omega',x)$, where
$\omega'\in\Omega$ is the repetition
of the word $\omega_1\cdots\omega_n$. For this reason, 
 properties of
random cycles may be analyzed through the analysis of  
the corresponding properties of periodic points of $R$.
Then, the level-2 large deviations upper bound for
suitably weighted periodic points of $R$ 
is available 
\cite{GelWol10,Kif94}. We convert this bound to 
 a samplewise almost-sure level-2 large deviations upper bound for weighted random cycles,
using a trick inspired by the proof of the level-1 large deviations upper bound 
due to Aimino, Nicol and Vaienti \cite[Proposition~3.14]{ANV15}. 
  This bound allows us to show that any weak* accumulation point of
  the sequence of measures in \eqref{random-c} is 
supported on the set of equilibrium states. Therefore,
the assumption of uniqueness in Theorem~A yields the convergence of the sequence. 
Integrating the samplewise large deviations upper bound over all samples
yields the convergence in Theorem~B.

\subsection{Statements of main results of product form}
 Theorem~A may be used to analyze time averages 
 $(1/n)\sum_{k=0}^{n-1}\varphi(T_\omega^k(x))$ of 
 a function $\varphi\colon X\to\mathbb R$ 
 along the random orbit of a 
 random cycle $x\in{\rm Fix}(T_\omega^n)$. However, it
 does not provide useful information on time averages
 of functions which depend on both $\omega$ and $x$.
Such functions with discontinuities naturally appear, for example, in random expansions of real numbers \cite{DK03,Da}.  
Therefore, the following version of Theorem~A has merit.
Let
$\delta_{(\omega,x)}^n$ denote the empirical measure $(1/n)\sum_{k=0}^{n-1}\delta_{R^k(\omega,x)}$.
For the definitions of acceptable functions on $\mathcal M(\Lambda)$
or $\Lambda$, see Definition~\ref{function-a}.

\begin{theoremc}
[Almost-sure level-2 weighted equidistribution of random cycles II]
Let $T_1,\ldots,T_N$ and $p$ be as in Theorem~A.
If $\tilde \varphi\colon\mathcal M(\Lambda)\to\mathbb R$ is an acceptable function, then
for $m_p$-almost every sample $\omega\in\Omega$
we have
\[
\lim_{n\to\infty}\frac{1}{Z_{\omega,n}}\sum_{x\in{\rm Fix}(T^n_\omega)}|(T_{\omega}^n)'x|^{-1}\tilde \varphi(\delta^n_{(\omega,x)})=
\tilde \varphi(m_p\otimes \lambda_p).\]
\end{theoremc}
Since the function $\tilde \varphi$ is allowed to be nonlinear (as $\tilde \varphi$ in Theorem~A), from Theorem~C one can deduce the convergence of various time averages of functions relative to random cycles.
We give three examples below,
inspired by the work of Olsen \cite[Section~1.1]{Ols03} on multifractal analysis.
\begin{corollary}[Almost-sure convergence of time averages relative to random cycles]\label{cor-10}
 Let $T_1,\ldots,T_N$ and $p$ be as in Theorem~A.
Then the following hold:

 \begin{itemize}
   
     \item[(a)]
if $\varphi\colon\Lambda\to\mathbb R$, $\psi\colon\Lambda\to\mathbb R$ are  acceptable  
  with $\inf \psi>0$,
  then for $m_p$-almost every $\omega\in\Omega$
we have
\[\lim_{n\to\infty}\frac{1}{Z_{\omega,n}}\sum_{x\in{\rm Fix}(T_\omega^n)}|(T^n_\omega)'(x)|^{-1}\frac{\sum_{k=0}^{n-1} \varphi(R^k(\omega,x))}{\sum_{k=0}^{n-1} \psi(R^k(\omega,x))}=\frac{\int \varphi d(m_p\otimes\lambda_p)}{\int \psi d(m_p\otimes\lambda_p)}.\]

\item[(b)]if $\varphi\colon\Lambda\to\mathbb R$, $\psi\colon\Lambda\to\mathbb R$ are  acceptable,
  then for $m_p$-almost every $\omega\in\Omega$
we have
\[\begin{split}\lim_{n\to\infty}\frac{1}{Z_{\omega,n}}\sum_{x\in{\rm Fix}(T_\omega^n)}|(T^n_\omega)'(x)|^{-1}\frac{1}{n^2}&\sum_{k=0}^{n-1} \varphi(R^k(\omega,x))\sum_{k=0}^{n-1} \psi(R^k(\omega,x))\\
&=\int \varphi d(m_p\otimes\lambda_p)\int \psi d(m_p\otimes\lambda_p).\end{split}\]

\item[(c)] 
if $\pi_i\colon \Lambda\to\mathbb R$ $(i=1,2)$ are continuous and  $g\colon\mathbb R\to\mathbb R$ is bounded continuous, then for $m_p$-almost every $\omega\in\Omega$ we have
\[\begin{split}\lim_{n\to\infty}\frac{1}{Z_{\omega,n}}\sum_{x\in{\rm Fix}(T_\omega^n)}|(T^n_\omega)'(x)|^{-1}\frac{1}{n^2}&\sum_{k_1,k_2=0}^{n-1} g(\pi_1(R^{k_1}(\omega,x))+\pi_2(R^{k_2}(\omega,x)))\\
&=\int g d((m_p\otimes\lambda_p)\circ \pi_1^{-1}*(m_p\otimes\lambda_p)\circ \pi_2^{-1}),\end{split}\]
where $*$ denotes the convolution.
\end{itemize}
\end{corollary}


A proof of Theorem~C also relies on large deviations for the skew product map $R$.
Using 
the trick of conversion as in the outline of the proof of Theorem~A, we obtain a samplewise upper bound, from which
 the desired convergence follows.

\subsection{Organization of the paper}

The rest of this paper consists of four sections.
In Section~2 we prove Theorems~A and B, and in Section~3 prove Theorem~C. 
In Section~4 we exhibit some examples of random dynamical systems 
to which our main results apply.
This includes 
the one
 introduced in \cite{DK03} that
generates expansions of real numbers with non-integer bases. Markov maps generating this system are 
uniformly expanding. To this system
we apply Theorem~C and Corollary~\ref{cor-10}, and obtain almost-sure convergences of average digital quantities in the expansions of random cycles
that do not follow from the 
application of the ergodic theorems of Birkhoff or Kakutani.
Also, we show that our main results are applicable to non-uniformly expanding Markov maps with common neutral fixed points, such as those introduced in
\cite{LSV99}.

\subsection{List of  measures}\label{list}
In addition to the measures $\tilde\xi_\omega^n$ in \eqref{random-c} 
and $\xi_n^\omega$ in \eqref{measure-xi},
we will use quite a few measures along the way:
$\tilde\mu_n$ \eqref{mutilde};
$\tilde\nu_n$ \eqref{nutilde}; 
$\tilde\mu_n^\omega$ \eqref{tildemuomega};
$\tilde\eta_{p,n}$ \eqref{tildeetapn}; $\eta_{p,n}$ \eqref{etapn};
$\mu_n$ \eqref{mun};
$\nu_n$ \eqref{nun};
  $\mu_{\omega,n}$ \eqref{muomegan}.
The tilder 
is attached to measures on $\mathcal M(\mathcal X)$  where
$\mathcal X=\Lambda$,
 $\Sigma$, $X$.


\section{Establishing weighted equidistributions of random cycles}
This section is devoted to the proofs of Theorems~A and B.
In Section~\ref{markov}
 we fix notations which permeate this paper, and
state the assumptions in the theorems.
After preliminaries in Section~\ref{prelim}, we introduce
 in Section~\ref{section4-1} a sequence $(\tilde\mu_{n})_{n=1}^\infty$ 
of measures
on $\mathcal M(\Lambda)$, and deduce the level-2 large deviations upper bound for closed sets (Proposition~\ref{annealed-ldpup}). 
In Section~\ref{subexp}, we provide a key estimate that allows us to convert the large deviations bound obtained in Section~\ref{section4-1} to a samplewise one almost surely. In Section~\ref{samplewise-2} we introduce another sequence 
$(\tilde\mu_{n}^\omega)_{n=1}^\infty$
of measures on $\mathcal M(\Lambda)$ as a samplewise version of $(\tilde\mu_{n})_{n=1}^\infty$, and deduce a samplewise level-2 upper bound
(Proposition~\ref{ldpup-q}). 
In Section~\ref{sample-2}
we establish the almost sure convergence of $(\tilde\mu_{n}^\omega)_{n=1}^\infty$.
In Section~\ref{pf-thma} we project this convergence result down
to a sequence of measures on the space $\mathcal M(X)$, and complete the proof of Theorem~A.  We prove an averaged convergence result in Section~\ref{av-r}, and
  Theorem~B in Section~\ref{pf-thmc}.

\subsection{Setup and assumptions}\label{markov}
Let $T_1,\ldots,T_N$ be non-uniformly expanding Markov maps on $X$.
For each $1\leq i\leq N$
let $\mathcal A(i)$ be a finite subset of $\mathbb N$ and 
  let $(J(a))_{a\in \mathcal A(i)}$ be the Markov partition for $T_i$. 
After re-indexing we may suppose that $i,j\in\{1,\ldots,N\}$, $i\neq j$ implies $\mathcal A(i)\cap\mathcal A(j)=\emptyset$.
We endow the set $\mathcal A=\bigcup_{i=1}^N \mathcal A(i)$ with the discrete topology, and
 denote by 
 $\mathcal A^{\mathbb N}$ the topological space that is
 the one-sided Cartesian product of $\mathcal A$, namely 
 \[\mathcal A^{\mathbb N}=\{\underline{a}=(a_n)_{n=1}^\infty\colon a_n\in\mathcal A\text{ for all }n\geq1\}.\]
For each $a\in \mathcal A$, let $i(a)$ denote the unique integer in $ \{1,\ldots, N\}$ such that $a\in \mathcal A(i(a))$.
  The sets
\[\varDelta(a)=\{\omega\in\Omega\colon \omega_1=i(a)\}\times J(a)\quad(a\in\mathcal A)\]
are pairwise disjoint subsets of $\Lambda$ (see FIGURE~1).
We will assume:
\begin{itemize}
    \item[(A1)] if $R(\varDelta(a))\cap{\rm int}(\varDelta(b))\neq\emptyset$ then ${\rm cl}(R(\varDelta(a)))\supset\varDelta(b)$.
\end{itemize}
If (A1) holds, we define a transition matrix $M=(m_{ab})_{a,b\in\mathcal A}$ by the rule $m_{ab}=1$ if $R(\varDelta(a))\cap{\rm int}(\varDelta(b))\neq\emptyset$ and $m_{ab}=0$ otherwise. Note that $M$ is an irreducible matrix.
We will assume:
\begin{itemize}
\item[(A2)] 
there exists an integer $n_0\geq1$ such that the matrix
$M^{n_0}$ has no zero entry. 
\end{itemize}
In other words, (A2) requires that
the Markov shift \[\Sigma=\{\underline{a}\in\mathcal A^{\mathbb N}\colon m_{a_na_{n+1}}=1\text{ for all }n\geq1\}\]
is topologically mixing. It also implies ${\rm Fix}(T_\omega^n)\neq\emptyset$ for $\omega\in\Omega$ and $n\geq n_0$. 

\begin{figure}
\begin{center}
\includegraphics[height=4cm,width=7cm]{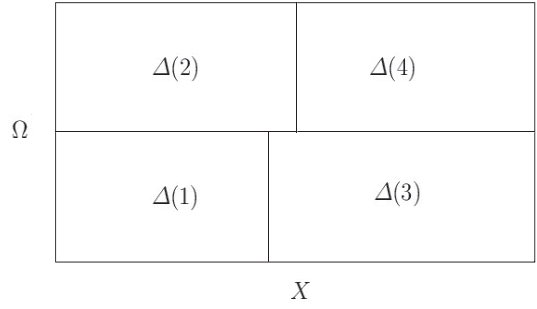}
\caption
{
The partition of the space $\Lambda=\Omega\times X$ in the case $N=2$, $\mathcal A(1)=\{1,3\}$, $\mathcal A(2)=\{2,4\}$.}
\end{center}
\end{figure}

By an {\it admissible word} we mean
any finite word of $\mathcal A$ appearing in an element of $\Sigma$.
Let 
$\mathcal A^n(\Sigma)$ denote the set of admissible words of word length $n$. 
For $a_1\cdots a_n\in\mathcal A^n(\Sigma)$ we set
\begin{equation}\label{remind1}\varDelta(a_1\cdots a_n)=\bigcap_{k=1}^n R^{-k+1}(\varDelta(a_k)).\end{equation}
As in Lemma~\ref{u-decay} below, 
the set 
$\bigcap_{n=1}^\infty{\rm cl}(\varDelta(a_1\cdots a_n))$
is a singleton for any $\underline{a}\in\Sigma$.
Hence, we can define a coding map $\pi\colon\Sigma\to \Lambda$ by
 \[\pi(\underline{a})\in\bigcap_{n=1}^\infty{\rm cl}(\varDelta(a_1\cdots a_n)).\]
 This map $\pi$ gives a semi-conjugacy between the skew product map $R$ and the left shift $\sigma$ on $\Sigma$.
 It is continuous 
 and one-to-one except on 
 the set
$E=\bigcup_{n=0}^\infty R^{-n}(\bigcup_{a\in \mathcal A}\partial\varDelta(a))$ where it is at most two-to-one.
Let $\Pi\colon\Lambda\to X$ denote the natural projection.
The upper-half of the following diagram commutes:
\begin{equation}\label{diagram}
  \begin{CD}
     \Sigma @>{\sigma}>> \Sigma \\
  @V{\pi}VV    @VV{\pi}V \\
     \Lambda  @>R>>  \Lambda\\
      @V{\Pi}VV     \\
     X  @>T_\omega>>  X.\\
  \end{CD}\end{equation}
We will assume:
\begin{itemize}
    \item[(A3)] For any $(\omega,x)\in\Lambda$ such that $R^n(\omega,x)=(\omega,x)$
    for some $n\geq1$,  
     $(\omega,x)\in\pi(\Sigma)$.
\end{itemize}

\noindent 
\begin{definition}\label{def1}We say $T_1,\ldots,T_N$  generate {\it a nice, topologically mixing skew product Markov map} if (A1) (A2) (A3) hold.\end{definition}

Condition (A3) means that any periodic point of $R$ is coded,
which is a mild assumption. It holds
if all $T_1,\ldots T_N$
are fully branched. 
What is at issue in (A3) is 
the existence of periodic points of $R$ which are contained in $E$. 
Since the growth of the number of these periodic points is at most of polynomial order, 
one can show that contributions from these periodic points are negligible
in the case all $T_1,\ldots,T_N$ are uniformly expanding. In particular,
(A3) is not needed in this case, see Section~\ref{random-beta} for more details.

Suppose that $T_1,\ldots,T_N$ generate a nice, topologically mixing skew product Markov map.
The remaining assumption in Theorem~A is concerned with the $N$-dimensional positive probability vector $p=(p_1,\ldots,p_N)$ and  stated in terms of 
the thermodynamic formalism.
   Define {\it a random geometric potential} $\phi=\phi_p\colon\Sigma \to\mathbb R$ by \[\phi(\underline{a})
=\log p_{\omega_1}-\log|T_{\omega_1}'(x)|,\]
where $\pi(\underline{a})=(\omega,x)$.
Let $\mathcal M(\Sigma,\sigma)$ denote the set of Borel probability measures on $\Sigma$
which are $\sigma$-invariant.
 For $\nu\in\mathcal M(\Sigma,\sigma)$ 
 let $h(\nu)$ denote the measure-theoretic entropy of $\nu$ relative to $\sigma$, and
   define a {\it free energy} $ F \colon \mathcal M(\Sigma,\sigma) \to \mathbb R$~by
\[ F(\nu)=h(\nu)+\int\phi d\nu.\]
 An {\it equilibrium state for the potential $\phi$} is a measure in $\mathcal M(\Sigma,\sigma)$ which maximizes the free energy. 
 A measure $\lambda\in\mathcal M(X)$ is {\it stationary} if 
$\lambda=\sum_{i=1}^Np_i\lambda\circ T_i^{-1}$. 
\begin{definition}\label{def2}
We say {\it the uniqueness of equilibrium state holds for $p$} if 
\begin{itemize}
 \item[(A4)] 
there exists a stationary measure $\lambda_p$ on $X$ such that $(m_p\otimes\lambda_p)\circ\pi$
is the unique equilibrium state
for the random geometric potential $\phi_p$.
 \end{itemize}
 \end{definition}
 \noindent  
  We have $R\circ \pi=\pi\circ\sigma$ on $\Sigma\setminus\pi^{-1}(E)$, 
and the restriction of $\pi$ to $\Sigma\setminus\pi^{-1}(E)$ has a continuous inverse. In particular,
$\pi$ maps Borel sets to Borel sets. This ensures that $(m_p\otimes\lambda_p)\circ\pi$ is indeed a $\sigma$-invariant Borel probability measure, and (A4) makes sense.

\subsection{Preliminary results}\label{prelim}

Let $T_1,\ldots,T_N$ be non-uniformly expanding Markov maps on $X$, and assume (A1).
We introduce further notations and prove some preliminary results.
For $a\in\mathcal A$
we set $f_a=T_{i(a)}|_{J(a)}.$
Compositions of these maps constitute branches of the random composition
$T_\omega^n$.
 For $n\geq2$ and an admissible word $a_1\cdots a_n\in \mathcal A^n(\Sigma)$, we define
\begin{equation}\label{remind2}f_{a_1\cdots a_n}=f_{a_n}\circ\cdots\circ f_{a_1}.\end{equation}
Let $f_{a_1\cdots a_n}^{-1}$ denote the inverse branch of the diffeomorphism $f_{a_1\cdots a_n}$, and define 
  \[J(a_1\cdots a_n)=f_{a_1\cdots a_n}^{-1}(X).\]

We denote by ${\rm Leb}$ the restriction of the
Lebesgue measure on $\mathbb R$ to $X$, and write $|J|={\rm Leb}(J)$
for a subinterval $J$ of $X$.
For a diffeomorphism $f$ from a bounded interval $J$ onto its image and
 $x,y\in J$, we put
\[D(f,x,y)=\left|\log|f'(x)|-\log|f'(y)|\right|.\]

\begin{lemma}\label{dist-basic}
If $T_1,\ldots,T_N$ are non-uniformly expanding and (A1) holds, there exists $C>0$ such that
for $n\geq1$, $a_1\cdots a_n\in \mathcal A^n(\Sigma)$ and $x,y\in J(a_1\cdots a_n)$ we have
\[D(f_{a_1\cdots a_n},x,y)\leq
C\sum_{j=1}^n
|f_{a_1\cdots a_j}(x)-f_{a_1\cdots a_j}(y)|^\tau.\]
\end{lemma}
\begin{proof}
Since $f_a'$ is $\tau$-H\"older continuous for each
$a\in \mathcal A$, there exists $C>0$ such that
for $a\in \mathcal A$ and $x,y\in J(a)$ we have
$D(f_{a_1},x,y)\leq C|f_a(x)-f_a(y)|^\tau,$ proving the desired inequality for $n=1$.
Iterating this argument yields the desired inequality for $n\geq2$.
\end{proof}

\begin{lemma}\label{u-decay}
If 
 $T_1,\ldots,T_N$ are non-uniformly expanding and (A1) holds, then
 \[\lim_{n\to\infty}\sup_{a_1\cdots a_n\in \mathcal A^n(\Sigma)}|J(a_1\cdots a_n)|=0.\]
In particular, for any $\underline{a}\in\Sigma$ the set 
$\bigcap_{n=1}^\infty{\rm cl}(\varDelta(a_1\cdots a_n))$
is a singleton.
\end{lemma}
\begin{proof}
Since 
 $T_1,\ldots,T_N$ are non-uniformly expanding, for any $\delta>0$ there exists a constant $c=c(\delta)>1$ such that any $1\leq i\leq N$ and
 any subinterval $J$ of $X$ that is contained in an element of the Markov partition for $T_i$ satisfying
$|J|\geq\delta$,
we have $|T_i(J)|\geq c|J|$.

If the desired convergence does not hold, 
 there exist $\epsilon>0$ and a strictly increasing sequence
 $(n_j)_{j=1}^\infty$ in $\mathbb N$ and a sequence $(\bold a_j)_{j=1}^\infty$ of admissible words such that 
 $\bold a_j\in \mathcal {A}^{n_j}(\Sigma)$ and $|J(\bold a_j)|>\epsilon$ for each $j\geq1$.
Then there exists a subinterval $J$ of $X$ with $|J|=\epsilon/2$
such that $J\subset J(\bold a_j)$ holds for infinitely many $j$. We have
$|f_{\bold a_j}(J)|\leq |f_{\bold a_j}(
J(\bold a_j)  )|\leq|X|$, 
while the property of the constant $c(\epsilon/2)$ yields
 $|f_{\bold a_j}(J)|\geq c(\epsilon/2)^{n_j}|J|$, which grows to infinity, a contradiction. Hence the first assertion of Lemma~\ref{u-decay} holds. The second assertion follows from the first one and the expansiveness of the left shift $\sigma$.
 \end{proof}

 For each $n\geq1$ and $a_1\cdots a_n\in \mathcal A^n(\Sigma)$, we define
 an $n$-cylinder
 \[\![a_1\cdots a_n]=\{\underline{b}\in \Sigma\colon b_k=a_k\text{ for }
 1\leq k\leq n\}.\]
 We write
$S_n\phi$ for the sum 
$\sum_{k=0}^{n-1}\phi\circ \sigma^k$,  and introduce an $n$-th variation
\[D_n(\phi)=\sup_{a_1\cdots a_n\in \mathcal A^n(\Sigma)}\sup_{\underline{b},\underline{c}\in [a_1\cdots a_n]}S_n\phi(\underline{b})-S_n\phi(\underline{c}).\]
Note that \[
\exp S_n\phi(\underline{a})=Q_p(\omega_1\cdots\omega_n)|(T_{\omega}^n)'x|^{-1},\]
where
$\pi(\underline{a})=(\omega,x)$.

\begin{lemma}\label{mild}
If 
 $T_1,\ldots,T_N$ are non-uniformly expanding and (A1) holds, then $D_n(\phi)=o(n)$ $(n\to\infty)$.
\end{lemma}
\begin{proof}
Let $a_1\cdots a_n\in \mathcal A^n(\Sigma)$.
By Lemma~\ref{dist-basic}, there exists $C>0$ such that for all $x,y\in J(a_1\cdots a_n)$ we have
\[\begin{split}D(f_{a_1\cdots a_n},x,y)&\leq C\left(|J(a_1\cdots a_n)|^\tau+\sum_{j=1}^{n-1}|f_{a_1\cdots a_j} (J(a_1\cdots a_n))|^\tau\right)\\
&\leq C\sum_{j=0}^{n-1}\sup_{a_1\cdots a_{n-j}\in \mathcal A^{n-j}(\Sigma)}|J(a_1\cdots a_{n-j})|^\tau.\end{split}\]
By Lemma~\ref{u-decay}, the last series is $o(n)$ and so Lemma~\ref{mild} holds.
\end{proof}


Following the thermodynamic formalism \cite{Bow75,Rue04},
we introduce a pressure
\[P(\phi)=\lim_{n\to\infty}\frac{1}{n}\log \sum_{a_1\cdots a_n\in\mathcal A^n(\Sigma)}\sup_{[a_1\cdots a_n]} \exp S_n\phi.\]
This limit exists and is finite.
The variational principle holds:
\begin{equation}\label{VP}P(\phi)=\sup\left\{F(\nu)\colon\nu\in\mathcal M(\Sigma,\sigma)\right\}.\end{equation}
 \begin{lemma}\label{exist-equi}
 If 
 $T_1,\ldots,T_N$ are non-uniformly expanding and (A1) holds, then
$P(\phi)=0$. 
\end{lemma}
\begin{proof}
For $n\geq1$ and $a_1\cdots a_n\in \mathcal A^n(\Sigma)$
we have
\[e^{-D_n(\phi)}\leq
\frac{(m_p\otimes{\rm Leb})(\varDelta(a_1\cdots a_n))}{
\sup_{[a_1\cdots a_n]} \exp S_n\phi}\leq e^{D_n(\phi)}.\]
We have
$\sum_{a_1\cdots a_n\in \mathcal A^n(\Sigma)}(m_p\otimes{\rm Leb})(\varDelta(a_1\cdots a_n))= |X|$, 
and Lemma~\ref{mild} gives $D_n(\phi)=o(n)$. 
We rearrange the double inequalities and
sum the result over all $a_1\cdots a_n\in \mathcal A^n(\Sigma)$. Then taking logarithms, dividing by $n$ and letting $n\to\infty$ yields $P(\phi)=0$. 
\end{proof}

\subsection{Level-2 upper bound for the skew product map}\label{section4-1}
Define a Borel probability measure
$\tilde\mu_n$ on $\mathcal M(\Lambda)$ by
  \begin{equation}\label{mutilde}\tilde\mu_n=\frac{1}{Z_{p,n}}\sum_{(\omega,x)\in {\rm Fix}(R^n)}Q_p(\omega_1\cdots\omega_n)
  |(T_\omega^n)'x|^{-1}\delta_{
 \delta_{(\omega,x)}^n}\quad(n=1,2,\ldots),\end{equation}
 where ${\rm Fix}(R^n)=\{(\omega,x)\in\Lambda\colon R^n(\omega,x)=(\omega,x)\}$ and
 $\delta_{(\omega,x)}^n$ denotes the empirical measure $(1/n)\sum_{k=0}^{n-1}\delta_{R^k(\omega,x)}$ in $\mathcal M(\Lambda)$.
 We also define a Borel probability measure
$\tilde\nu_n$ on
 $\mathcal M(\Sigma)$ by
\begin{equation}\label{nutilde}\tilde\nu_n=\left(\sum_{\underline{a}\in{\rm Fix}(\sigma^n)} \exp S_n\phi(\underline{a})\right)^{-1}\sum_{\underline{a}\in{\rm Fix}(\sigma^n)} \exp S_n\phi(\underline{a})\delta_{\delta_{\underline{a}}^n}\quad(n=1,2,\ldots),\end{equation}
where ${\rm Fix}(\sigma^n)=\{\underline{a}\in\Sigma\colon\sigma^n\underline{a}=\underline{a}\}$ and $\delta_{\underline{a}}^n=(1/n)\sum_{k=0}^{n-1}\delta_{\sigma^k\underline{a}}$. 

We compare the two measures through the continuous map
 $\pi$.
The push-forward $\pi_*\colon \nu\in\mathcal M(\Sigma)\mapsto \nu\circ\pi^{-1}\in\mathcal M(\Lambda)$ is continuous. The measure $\tilde\nu_n\circ\pi_*^{-1}$ is almost isomorphic
to $\tilde\mu_n$, up to elements of ${\rm Fix}(R^n)$ contained in the boundary points of the sets $\varDelta(a)$.
More precisely, they are related as follows.
\begin{lemma}\label{relation}
Assume (A3).
For any $n\geq1$ and
 any Borel subset $\mathcal B$ of $\mathcal M(\Lambda)$, we have
$\tilde\mu_n(\mathcal B)\leq2\tilde\nu_n\circ\pi_*^{-1}(\mathcal B).$
\end{lemma}
\begin{proof}
Both measures are supported on finite sets.
If $\underline{a}\in\Sigma$, $(\omega,x)\in\Lambda$ and
$\pi(\underline{a})=(\omega,x)$, then we have $S_n\phi(\underline{a})=Q_p(\omega_1\cdots\omega_n)|(T_{\omega}^n)'x|^{-1}$.
  Since $\pi$ is at most two-to-one and ${\rm Fix}(R^n)\subset\pi({\rm Fix}(\sigma^n))$ from (A3), for each $(\omega,x)\in{\rm Fix}(R^n)$ there exists at least one and at most two points in $\Sigma$ which are mapped to $(\omega,x)$ by $\pi.$
  Hence the desired inequality holds.
  \end{proof}
 
 As a consequence, the large deviations upper bound for $(\tilde\nu_n)_{n=1}^\infty$ implies that for $(\tilde\mu_n)_{n=1}^\infty$.
The former is controlled by
the function $ I_\Sigma \colon \mathcal M(\Sigma) \to (-\infty,\infty]$ given by
\begin{equation}\label{r-f} I_\Sigma(\nu)
=
\begin{cases}-F(\nu) &\text{ if $\nu\in\mathcal M(\Sigma,\sigma)$},\\
\infty&\text{ otherwise.}\end{cases}\end{equation}
From the variational principle \eqref{VP} and
$P(\phi)=0$ in Lemma~\ref{exist-equi},
 $I_\Sigma$ is a non-negative function.
We define $I_{\Lambda}\colon \mathcal M(\Lambda)\to[0,\infty]$
by
\[I_{\Lambda}(\mu)=\inf\{I_\Sigma(\nu)\colon\nu\in\mathcal M(\Sigma), \nu\circ\pi^{-1}=\mu\}.\]

 \begin{prop}\label{annealed-ldpup}
 Let
 $T_1,\ldots,T_N$ be non-uniformly expanding Markov maps on $X$ generating a nice, topologically mixing skew product Markov map. Then
 $I_{\Lambda}$ is lower semicontinuous and
for any closed subset $\mathcal C$ of $\mathcal M(\Lambda)$,
\[\limsup_{n\to\infty}\frac{1}{n}\log \tilde\mu_n(\mathcal C)\leq -\inf_{\mathcal C} I_{\Lambda}.
\]
\end{prop}
\begin{proof}
Since 
the left shift $\sigma$ is expansive and $\Sigma$ is compact,
 the measure-theoretic entropy 
is upper semicontinuous on $\mathcal M(\Sigma,\sigma)$. Moreover, $\phi$ is continuous with respect to the shift metric on $\Sigma$ and $\mathcal M(\Sigma,\sigma)$ is a closed subset of $\mathcal M(\Sigma)$. Hence, $I_\Sigma$ is  lower semicontinuous.
Since the coding map $\pi$ is continuous, $I_{\Lambda}$ is lower semicontinuous.


By \cite[Theorem~2.1]{Kif94} when
 $\phi$ is H\"older continuous, and 
 by \cite[Theorem~6]{GelWol10} when  $\phi$ is merely continuous, 
for any closed subset $\mathcal C$ of $\mathcal M(\Sigma)$ we have the large deviations upper bound
\begin{equation}\label{ldupest}\limsup_{n\to\infty}\frac{1}{n}\log \tilde\nu_n(\mathcal C)\leq -\inf_{\mathcal C} I_\Sigma.
\end{equation}
 From this bound and Lemma~\ref{relation}, 
for any closed subset $\mathcal C$ of $\mathcal M(\Lambda)$ we obtain
\[\limsup_{n\to\infty}\frac{1}{n}\log \tilde\mu_n(\mathcal C)\leq\limsup_{n\to\infty}\frac{1}{n}\log \tilde\nu_n(\pi_*^{-1}(\mathcal C))\leq -\inf_{\pi_*^{-1}(\mathcal C)} I_\Sigma=-\inf_{\mathcal C} I_{\Lambda},
\]
as required.
\end{proof}
\subsection{Sub-exponential bound on the sum of derivatives}\label{subexp}
The following lemma gives a bound on the normalizing constant
\[Z_{\omega,n}=\sum_{x\in{\rm Fix}(T^n_\omega)}|(T_{\omega}^n)'x|^{-1},\]
which is a key needed to convert the level-2 upper bound in Proposition~\ref{annealed-ldpup} to samplewise ones.
 \begin{lemma}\label{b-d}
 If 
 $T_1,\ldots,T_N$ are non-uniformly expanding and (A1) (A2) hold, then
 there exists a sequence $(c_n)_{n=1}^\infty$ of positive reals  such that $c_n=o(n)$ as $n\to\infty$, and
for $\omega\in\Omega$ and $n\geq n_0$, we have
$e^{-c_n}\leq Z_{\omega,n}\leq e^{c_n}.$
\end{lemma}
\begin{proof}
Recall the relation $a\in\mathcal A(i(a))$ for $a\in\mathcal A$.
Set $c=\inf_{a\in \mathcal A} |T_{i(a)}(J(a))|$. For $\omega\in\Omega$ and $n\geq1$ we have
  \begin{equation}\label{b-d1} \inf_{a_1\cdots a_n\in\mathcal A^n(\Sigma)}|T_\omega^n(
  J(a_1\cdots a_n))|\geq c>0.\end{equation}
    By the mean value theorem,
  for each $a_1\cdots a_n\in \mathcal A^n(\Sigma)$ there exists $x(a_1\cdots a_n)\in
  J(a_1\cdots a_n)$ such that
  \begin{equation}\label{b-d2}|(T_\omega^n)'x(a_1\cdots a_n)|^{-1}=\frac{|J(a_1\cdots a_n)|}{|T_\omega^n(J(a_1\cdots a_n))|}.\end{equation}
  Let $
  \mathcal A^n(\Sigma,\omega)=\{a_1\cdots a_n\in \mathcal A^n(\Sigma)\colon i({a}_k)=\omega_k\text{ for }1\leq k\leq n\}$. 
  Since all the maps are non-uniformly expanding,
   for each $a_1\cdots a_n\in \mathcal A^n(\Sigma,\omega)$ the interval
 $J(a_1\cdots a_n)$ contains at most one point from ${\rm Fix}(T_\omega^n)$. Using \eqref{b-d1} and \eqref{b-d2}, on the one hand we have
\[\begin{split}Z_{\omega,n}&\leq e^{D_n(\phi)}\sum_{a_1\cdots a_n\in \mathcal A^n(\Sigma,\omega)   }|(T_\omega^n)'
x(a_1\cdots a_n)|^{-1}\\
&\leq c^{-1}e^{D_n(\phi)}\sum_{a_1\cdots a_n\in \mathcal A^n(\Sigma,\omega)   } |J(a_1\cdots a_n)|\leq c^{-1}e^{D_n(\phi)}.\end{split}\]

On the other hand,
let $n\geq n_0$ where $n_0$ is the integer in (A2). 
For $a_1\cdots a_n\in\mathcal A^n(\Sigma)$ we have
${\rm cl}(T_\omega^n(J(a_1\cdots a_n)))=X$. Hence,
if 
$J(a_1\cdots a_n)$ does not intersect ${\rm Fix}(T_\omega^n)$ then
${\rm cl}(J(a_1\cdots a_n))$ contains one of the boundary points of $X$. 
This implies
\begin{equation}\label{cardinality}
    \#\{a_1\cdots a_n\in \mathcal A^n(\Sigma,\omega)\colon
 J(a_1\cdots a_n)\cap{\rm Fix}(T_\omega^n)\neq\emptyset\}\leq 2.
\end{equation}
Using \eqref{b-d2} again we have
\[\begin{split}Z_{\omega,n}&\geq e^{-D_n(\phi)}\sum_{\stackrel{a_1\cdots a_n\in \mathcal A^n(\Sigma,\omega)  }{J(a_1\cdots a_n)\cap{\rm Fix}(T_\omega^n)\neq\emptyset}}|(T_\omega^n)'
x(a_1\cdots a_n)|^{-1}\\
&\geq \frac{e^{-D_n(\phi)}}{|X|}\sum_{\stackrel{
a_1\cdots a_n\in \mathcal A^n(\Sigma,\omega)  
 }{J(a_1\cdots a_n)\cap{\rm Fix}(T_\omega^n)\neq\emptyset}} |J(a_1\cdots a_n)|\geq \frac{e^{-D_n(\phi)}}{2|X|},\end{split}\]
where the last inequality holds for sufficiently large $n$ by Lemma~\ref{u-decay} and \eqref{cardinality}.
Set $c_n=D_n(\phi)+\log (2|X|/c)$. Lemma~\ref{mild} yields $c_n=o(n)$ as required.
\end{proof}

 \subsection{Samplewise level-2 upper bound}\label{samplewise-2}
 For $\omega\in\Omega$, we define a Borel probability measure
 $\tilde\mu_n^\omega$ on
 $\mathcal M(\Lambda)$ by
 \begin{equation}\label{tildemuomega}\tilde\mu_n^\omega=\frac{1}{Z_{\omega,n}}\sum_{x\in{\rm Fix}(T^n_\omega)}|(T_{\omega}^n)'x|^{-1}\delta_{\delta_{(\omega,x)}^n}\quad(n=1,2,\ldots),\end{equation}
 which is a samplewise version of $\tilde\mu_n$ in \eqref{mutilde}. 


\begin{prop}\label{ldpup-q}
Let
 $T_1,\ldots,T_N$ be non-uniformly expanding Markov maps on $X$ generating a nice, topologically mixing skew product Markov map. For $m_p$-almost every $\omega\in\Omega$ and
  any closed subset $\mathcal C$ of $\mathcal M(\Lambda)$, we have
\begin{equation}\label{ldpup-lem}\limsup_{n\to\infty}\frac{1}{n}\log \tilde\mu_{n}^\omega(\mathcal C)\leq -\inf_{\mathcal C} I_{\Lambda}.\end{equation}
\end{prop}
\begin{proof}
We claim that it is enough to show \eqref{ldpup-lem}
for each closed subset $\mathcal C$ of $\mathcal M(\Lambda)$ 
and any sample $\omega$ contained in a Borel set $\Omega_{\mathcal C}$ with full $m_p$-measure. 
To show this claim, 
we fix a metric which generates the weak* topology on $\mathcal M(\Lambda)$, and fix its
  countable dense subset $\mathcal D$.
For $\nu\in \mathcal D$ and $r>0$ 
let $\mathcal B_r(\nu)$ denote the closed ball of radius $r$ about $\nu$.
From \eqref{ldpup-lem},
 the Borel set $\bigcap_{\nu\in\mathcal D, r\in\mathbb Q,r>0}\Omega_{\mathcal B_r(\nu)}$ has full $m_p$-measure, and if $\omega\in\Omega$ is contained in this set, then for 
$\nu\in \mathcal D$ and $r\in\mathbb Q$ with $r>0$ we have
\[\limsup_{n\to\infty}\frac{1}{n}\log \tilde\mu_{n}^\omega(\mathcal B_r(\nu))\leq -\inf_{\mathcal B_r(\nu)}I_{\Lambda}.\]
Let $\mathcal C$ be a non-empty closed subset of $\mathcal M(\Lambda)$.
Let $\mathcal G$ be an open subset of $\mathcal M(\Lambda)$ which contains $\mathcal C$.
Since $\mathcal C$ is compact,
there exists a finite subset $\{\nu_1,\ldots,\nu_s\}$ of $\mathcal D$ and $r_1,\ldots,r_s\in\mathbb Q$ such that
$\mathcal C\subset \bigcup_{j=1}^s\mathcal B_{r_j}(\nu_j)\subset\mathcal G$.
Then
\[\begin{split}\limsup_{n\to\infty}\frac{1}{n}\log \tilde\mu_{n}^\omega(\mathcal C)&\leq\max_{1\leq j\leq s}\limsup_{n\to\infty}\frac{1}{n}\log \tilde\mu_n^{\omega}(\mathcal B_{r_j}(\nu_j))\\
&\leq\max_{1\leq j\leq s}\left( -\inf_{\mathcal B_{r_j}(\nu_j)}I_{\Lambda}\right)\leq-\inf_{\mathcal G} I_{\Lambda}.\end{split}\]
Since $\mathcal G$ is an arbitrary open set containing $\mathcal C$ and $I_{\Lambda}$ is lower semicontinuous,
we obtain \eqref{ldpup-lem}.

In what follows we assume $0<\inf_{\mathcal C}I_{\Lambda}<\infty$, for otherwise
\eqref{ldpup-lem} clearly holds.
 Using the definitions of $\tilde\mu_n$ in \eqref{mutilde} and $\tilde\mu_n^\omega$ in \eqref{tildemuomega} and
 the formula
\begin{equation}\label{zn}Z_{p,n}=\sum_{\omega_1\cdots \omega_n\in \{1,\ldots,N\}^n}Q_p(\omega_1\cdots\omega_n) Z_{\omega,n},\end{equation}
 and then Lemma~\ref{b-d} we have 
\[\begin{split}
\tilde\mu_n(\mathcal C)&=\frac{1}{Z_{p,n}}\sum_{\stackrel{(\omega,x)\in{\rm Fix}(R^n)}{\delta_{(\omega,x)}^n\in\mathcal C}}Q_p(\omega_1\cdots\omega_n)|(T_\omega^n)'x|^{-1}\\
&=\int \sum_{\stackrel{x\in{\rm Fix}(T^n_\omega)}{\delta_{(\omega,x)}^n\in\mathcal C}}|(T_{\omega}^n)'x|^{-1} dm_p(\omega) \Big/\int Z_{\omega',n}  dm_p(\omega') \\
&=\int \tilde\mu_n^\omega(\mathcal C)\left({Z_{\omega,n} \Big/ \int
Z_{\omega',n} dm_p(\omega')}\right)dm_p(\omega)\geq e^{-2c_n} \int\tilde\mu_n^\omega(\mathcal C)\ dm_p(\omega).
\end{split}\]
For $\epsilon\in(0,1)$ and $n\geq1$, set 
\[\Omega_{\epsilon,n}=\left\{\omega\in\Omega\colon \tilde\mu_n^\omega(\mathcal C)\geq \exp\left(-n(1-\epsilon)\inf_{\mathcal C} I_{\Lambda}\right)\right\}.\]
Then Markov's inequality yields
\[
\begin{split}m_p(\Omega_{\epsilon,n})
&\leq \exp\left(n(1-\epsilon)\inf_{\mathcal C} I_{\Lambda}\right)\int\tilde\mu_n^\omega(\mathcal C)d m_p(\omega)\\
&\leq e^{2c_n}\exp\left(n(1-\epsilon)\inf_{\mathcal C}I_{\Lambda}\right)\tilde\mu_n(\mathcal C).\end{split}\]
By Proposition~\ref{annealed-ldpup},
$m_p(\Omega_{\epsilon,n})$ decays exponentially as $n$ increases.
By Borel-Cantelli's lemma, the number of those $n\geq1$
for which the inequality $\tilde\mu_n^\omega(\mathcal C)\geq \exp(-n(1-\epsilon)\inf_{\mathcal C} I_{\Lambda})$ holds is finite for $m_p$-almost every $\omega\in\Omega.$
Since $\epsilon\in(0,1)$ is arbitrary, we obtain 
\eqref{ldpup-lem} for $m_p$-almost every $\omega\in\Omega$. 
\end{proof}


\subsection{Samplewise convergence of measures}\label{sample-2}
From the samplewise level-2 upper bound in Section~\ref{samplewise-2} we deduce the following samplewise weak* convergence.
\begin{prop}\label{abstract}
Let
 $T_1,\ldots,T_N$ be non-uniformly expanding Markov maps on $X$ generating a nice, topologically mixing skew product Markov map. 
If the uniqueness of equilibrium state holds
 for a probability vector $p$,
 then
 for $m_p$-almost every $\omega\in\Omega$, $(\tilde\mu_n^\omega)_{n=1}^\infty$ converges to $\delta_{m_p\otimes\lambda_p}$ in the weak* topology as $n\to\infty$.
\end{prop}

\begin{proof}
Let $\omega\in\Omega$ be as in Proposition~\ref{ldpup-q}.
Let $(\tilde{\mu}_{n_j}^\omega)_{j=1}^\infty$
be an arbitrary convergent subsequence of $(\tilde\mu_n^\omega)_{n=1}^\infty$ with the limit measure $\tilde\mu^\omega$. It suffices to show 
 $\tilde\mu^\omega=\delta_{m_p\otimes\lambda_p}$.

 \begin{lemma}\label{minimizer} 
We have $I_{\Lambda}(\mu)=0$ if and only if $\mu=m_p\otimes\lambda_p$.\end{lemma}
\begin{proof}By (A4) we have $F((m_p\otimes\lambda_p)\circ\pi)=P(\phi)$, and
Lemma~\ref{exist-equi} gives $P(\phi)=0$.
Hence,
$I_{\Lambda}(m_p\otimes\lambda_p)=0$.
 Conversely, let $\mu\in\mathcal M(\Lambda)$ satisfy $I_{\Lambda}(\mu)=0$. Then there exists a sequence $(\nu_n)_{n=1}^\infty$ in $\mathcal M(\Sigma)$ such that $\mu=\nu_n\circ\pi^{-1}$ and $\lim_{n\to\infty}I_\Sigma(\nu_n)=0$. By compactness, it has
   a limit point $\nu_\infty$.
 Since $I_\Sigma$ is lower semicontinuous we have $I_\Sigma(\nu_\infty)=0$,
 and so $F(\nu_\infty)=0$.
 The uniqueness in (A4) yields $\nu_\infty=(m_p\otimes\lambda_p)\circ\pi$, and thus $\mu=m_p\otimes\lambda_p$. 
\end{proof}

We fix a metric which generates the weak* topology on $\mathcal M(\Lambda)$. 
For $\epsilon>0$ let $\mathcal L_\epsilon=\{\mu\in\mathcal M(\Lambda)\colon I_{\Lambda}(\mu)\leq\epsilon\}$.
Since $I_{\Lambda}$ is lower semicontinuous,
$\mathcal L_\epsilon$ is a closed set. 
Since $\mathcal M(\Lambda)$ is compact, so is $\mathcal L_\epsilon$. 
Let $\nu\in\mathcal M(\Lambda)\setminus\{m_p\otimes\lambda_p\}$. Lemma~\ref{minimizer} gives $I_{\Lambda}(\nu)>0$.
 Take $r>0$ such that the closed ball $\mathcal B_r(\nu)$ of radius $r$ about $\nu$ 
 does not intersect $\mathcal L_{I_{\Lambda}(\nu)/2}$.
By the convergence
$\lim_{j\to\infty}\tilde{\mu}_{n_j}^\omega=\tilde\mu^\omega$
 and Proposition~\ref{ldpup-q},
we have
\[\begin{split}\tilde\mu^\omega({\rm int}(\mathcal B_r(\nu)))&\leq\liminf_{j\to\infty}\tilde\mu_{n_j}^\omega({\rm int}(\mathcal B_r(\nu)))\leq\limsup_{j\to\infty}\tilde\mu_{n_j}^\omega(\mathcal B_r(\nu))\\
&\leq\limsup_{j\to\infty}\exp(-I_{\Lambda}(\nu) n_j/2)=0.\end{split}\]
Hence, the support of $\tilde\mu^\omega$ does not contain $\nu$. Since
 $\nu$ is an arbitrary element of
 $\mathcal M(\Lambda)\setminus\{m_p\otimes\lambda_p\}$, we obtain
 $\tilde\mu^\omega=\delta_{m_p\otimes\lambda_p}$ as required.
\end{proof}

\subsection{Proof of Theorem~A}\label{pf-thma}

Since the projection $\Pi\colon\Lambda\to X$
is continuous, the push-forward $\Pi_*\colon\mu\in\mathcal M(\Lambda)\mapsto\mu\circ\Pi^{-1}\in\mathcal M(X)$ is continuous.
Another push-forward 
$\Pi_{**}\colon\tilde\mu\in \mathcal M(\mathcal M(\Lambda))\mapsto\tilde\mu\circ(\Pi_*)^{-1}\in\mathcal M(\mathcal M(X))$ is continuous too.
Note that
$\Pi_*(\mu)=\nu$ implies
$\Pi_{**}(\delta_{\mu})=\delta_{\nu}.$
In particular, $\Pi_{**}(\delta_{m_p\otimes\lambda_p})=\delta_{\lambda_p}$ and
$\Pi_{**}(\delta_{\delta_{(\omega,x)}^n})=\delta_{\delta_x^{\omega,n}}$, and the latter yields
 $\Pi_{**}(\tilde\mu_n^\omega)
=\tilde\xi_n^\omega.$
By Proposition~\ref{abstract},  for $m_p$-almost every $\omega\in\Omega$
we have
$\tilde\mu_n^\omega\to\delta_{m_p\otimes\lambda_p}$ in the weak* topology as $n\to\infty$.
Since $\Pi_{**}$ is continuous, we obtain
$\tilde\xi_n^\omega\to\delta_{\lambda_p}$ in the weak* topology as $n\to\infty$.
\qed

\subsection{Averaged result}\label{av-r}
As a corollary to Theorem~A, we obtain an averaged result 
over all samples.
By Riesz's representation theorem, for each positive probability vector $p$ and $n\geq1$
there is a unique Borel probability measure 
 $\tilde\eta_{p,n}$ on $\mathcal M(X)$ that satisfies
\begin{equation}\label{tildeetapn}\int\tilde \varphi d\tilde\eta_{p,n}=\int dm_p(\omega)\int\tilde \varphi d\tilde\xi_{n}^{\omega}\quad\text{for any continuous }\tilde \varphi\colon\mathcal M(X)\to\mathbb R.\end{equation}
Also,
there is a unique Borel probability measure 
$\eta_{p,n}$ on $X$ that satisfies
\begin{equation}\label{etapn}\int \varphi d\eta_{p,n}=\int dm_p(\omega)\int \varphi d\xi_{n}^{\omega}\quad\text{for any continuous }\varphi\colon X\to\mathbb R.\end{equation}

\begin{corollary}[Averaged weighted equidistribution of random cycles II]\label{thmb}

Let 

\noindent
$T_1,\ldots,T_N$ and $p$ be as in Theorem~A.
 Then 
 $(\tilde{\eta}_{p,n})_{n=1}^\infty$ converges to
$\delta_{\lambda_p}$ in the weak* topology as $n\to\infty$ and $(\eta_{p,n})_{n=1}^\infty$ converges to $\lambda_p$ in the weak* topology as $n\to\infty$.
\end{corollary}
\begin{proof}
Let $\tilde \varphi\colon\mathcal M(X)\to\mathbb R$ be an arbitrary continuous function.
By \eqref{tildeetapn} and
Theorem~A, $\lim_{n\to\infty}\int\tilde \varphi d\tilde\xi_{n}^\omega=\int\tilde \varphi d\delta_{\lambda_p}$ holds
for $m_p$-almost every $\omega\in \Omega$.
The dominated convergence theorem gives $\lim_{n\to\infty}\int\tilde \varphi d\eta_{p,n}=\int\tilde \varphi d\delta_{\lambda_p}$. 
Since $\tilde\varphi$ is an arbitrary continuous function on $\mathcal M(X)$, 
we obtain
 $\tilde\eta_{p,n}\to\delta_{\lambda_p}$ in the weak* topology 
 as $n\to\infty$. 
The second 
convergence in the corollary follows from the first one.
\end{proof}

\subsection{Proof of Theorem~B}\label{pf-thmc}
Let $p$ be an $N$-dimensional positive probability vector. Define a Borel probability measure
$\tilde{\zeta}_{p,n}$ on $\mathcal M(X)$ by
\[\tilde{\zeta}_{p,n}=\frac{1}{Z_{p,n}}\sum_{\omega_1\cdots \omega_n\in\{1,\ldots,N\}^n}Q_{p}(\omega_1\cdots\omega_n)\sum_{x\in{\rm Fix}(T_{\omega}^n)}|(T_{\omega}^n)'x|^{-1}\delta_{\delta_x^{\omega,n}}\quad(n=1,2,\ldots).\]
It is enough to show that $\tilde\zeta_{p,n}$ converges to $\delta_{\lambda_p}$
in the weak* topology as $n\to\infty$.
Let $\mathcal C$ be an arbitrary closed subset of $\mathcal M(X)$ not containing $\lambda_p$. 
From the proof of Proposition~\ref{abstract},
  there exists $\alpha>0$ such that
$\tilde\xi_n^\omega(\mathcal C)=\tilde\mu_{n}^\omega(\Pi_{*}^{-1}(\mathcal C))\leq e^{-\alpha n}$ for all sufficiently large $n$.
Since
$\tilde\eta_{p,n}(\mathcal C)=\int\tilde\xi_{n}^\omega(\mathcal C)dm_p(\omega)$ by \eqref{tildeetapn},
the dominated convergence theorem gives
$\lim_{n\to\infty}e^{\alpha n/2} \tilde\eta_{p,n}(\mathcal C)=\lim_{n\to\infty}\int e^{\alpha n/2}\tilde\xi_{n}^\omega(\mathcal C)dm_p(\omega)=0$, and so
 $\lim_{n\to\infty}\tilde\eta_{p,n}(\mathcal C)=0$.

Using Lemma~\ref{b-d} we have
\[\begin{split}
\tilde\zeta_{p,n}(\mathcal C)&=\frac{1}{Z_{p,n}}\sum_{\omega_1\cdots\omega_n\in\{1,\ldots,N\}^n}\sum_{\stackrel{x\in{\rm Fix}(T^n_\omega)}{\delta_{x}^{\omega,n}\in\mathcal C}}Q_p(\omega_1\cdots\omega_n)|(T_\omega^n)'x|^{-1}\\
&=\int\sum_{\stackrel{x\in{\rm Fix}(T_\omega^n)}{\delta_{x}^{\omega,n}\in\mathcal C}}|(T_\omega^n)'x|^{-1} dm_p(\omega)\Big/\int Z_{\omega',n} dm_p(\omega')\\
&\leq e^{c_n}\int\sum_{\stackrel{x\in{\rm Fix}(T_\omega^n)}{\delta_{x}^{\omega,n}\in\mathcal C}}|(T_\omega^n)'x|^{-1} dm_p(\omega)\Big/\min_{\omega'\in\Omega}Z_{\omega',n}\\
&\leq e^{2c_n}\int \frac{1}{Z_{\omega,n}}\sum_{\stackrel{x\in{\rm Fix}(T_\omega^n)}{\delta_{x}^{\omega,n}\in\mathcal C}}|(T_\omega^n)'x|^{-1} dm_p(\omega)= e^{2c_n}\tilde\eta_{p,n}(\mathcal C),\end{split}\]
 and therefore
$\lim_{n\to\infty}\tilde\zeta_{p,n}(\mathcal C)=0.$
Since $\mathcal C$ is an arbitrary closed subset of $\mathcal M(X)$
not containing $\lambda_p$, it follows that
 $\tilde\zeta_{p,n}\to\lambda_p$
 in the weak* topology as $n\to\infty$.\qed

\section{Random cycles with weight functions on the product space}
In this section we slightly extend ideas in Section~2 to prove Theorem~C.
In Section~\ref{est1} we introduce a sequence of measures on the product space $\Lambda$, and establish the level-1 large deviations upper bound for acceptable functions.
In Section~\ref{sample-1}
we deduce a samplewise almost-sure level-1 large deviations upper bound.
In Section~\ref{est2} we complete the proof of Theorem~C.

\subsection{Level-1 upper bound for the skew product map}\label{est1}
We define a Borel probability measure $\mu_n$ on $\Lambda$ by
\begin{equation}\label{mun}\mu_n=\frac{1}{Z_{p,n}}\sum_{(\omega,x)\in {\rm Fix}(R^n)}Q_p(\omega_1\cdots\omega_n) |(T_\omega^n)'x|^{-1}
 \delta_{(\omega,x)}\quad(n=1,2,\ldots).\end{equation}
\begin{definition}\label{function-a}
We say a function $\tilde \varphi\colon\mathcal M(\Lambda)\to\mathbb R$ is {\it acceptable} if $\nu\in\mathcal M(\Sigma)\mapsto
\tilde \varphi(\nu\circ\pi^{-1})$ is continuous.
We say a function $\varphi\colon\Lambda\to\mathbb R$ is {\it acceptable} if
$\varphi\circ\pi$ is continuous.
\end{definition}
\begin{remark}
{\rm The class of acceptable functions is strictly larger than that of continuous ones. Hence, Theorem~C is not a consequence of Proposition~\ref{abstract}.}\end{remark}

Let $\tilde \varphi$ be an acceptable function on $\mathcal M(\Lambda)$.
By means of large deviations,
we estimate the exponential decay rate of the set
\[A_{n}(\tilde \varphi,K)=\left\{(\omega,x)\in\Lambda\colon 
\tilde \varphi(\delta_{(\omega,x)}^{n})\in K\right\}.\]
Recall the definition of $I_\Sigma$ in \eqref{r-f}, and
define a rate function $I_{\Sigma,\tilde \varphi}\colon\mathbb R\to[0,\infty]$ by
\[I_{\Sigma,\tilde \varphi}(\alpha)=\inf\left\{I_\Sigma(\nu)\colon\nu\in\mathcal M(\Sigma),\ \tilde \varphi(\nu\circ\pi^{-1})=\alpha \right\}.\]

\begin{prop}\label{level-1LDP}
Let
 $T_1,\ldots,T_N$ be
  non-uniformly expanding Markov maps on $X$ generating a nice, topologically mixing skew product Markov map.
Let $\tilde \varphi\colon\mathcal M(\Lambda)\to\mathbb R$ be acceptable.
For any closed subset $K$ of $\mathbb R$ we have
\[\limsup_{n\to\infty}\frac{1}{n}\log \mu_{n}(A_n(\tilde \varphi,K))\leq -\inf_K I_{\Sigma,\tilde \varphi}.\]
\end{prop}
\begin{proof}
We proceed much in parallel to the proof of Proposition~\ref{annealed-ldpup},
comparing $\mu_n$ in \eqref{mun} and a Borel probability measure
 $\nu_n$ on $\Sigma$ given by
\begin{equation}\label{nun}\nu_n=\left(\sum_{\underline{a}\in{\rm Fix}(\sigma^n)} \exp S_n\phi(\underline{a})\right)^{-1}\sum_{\underline{a}\in{\rm Fix}(\sigma^n)} \exp S_n\phi(\underline{a})\delta_{\underline{a}}.\end{equation}
Recall the definition of $\tilde\nu_n$ in \eqref{nutilde}, and
notice the identity
\begin{equation}\label{1LDP-1}\tilde\nu_n\left\{\nu\in\mathcal M(\Sigma)\colon \tilde \varphi(\nu\circ\pi^{-1})\in K\right\}=\nu_n\left\{\underline{a}\in\Sigma\colon \pi(\underline{a})\in A_{n}(\tilde \varphi,K)\right\}.\end{equation}
Since 
$\tilde \varphi$ is acceptable and $K$ is closed,
the contraction principle \cite{Ell85} applied to the upper bound \eqref{ldupest} yields
\begin{equation}\label{1LDP-2}\limsup_{n\to\infty}\frac{1}{n}\log  \tilde\nu_n\left\{\nu\in\mathcal M(\Sigma)\colon \tilde \varphi(\nu\circ\pi^{-1})\in K\right\}\leq -\inf_{K} I_{\Sigma,\tilde\varphi}.
\end{equation}
Plugging \eqref{1LDP-1} into the left-hand side of \eqref{1LDP-2} 
and 
combining the result with the inequality 
 $
\mu_{n}(A_n(\tilde \varphi,K))\leq
 2\nu_n\left\{\underline{a}\in\Sigma\colon \pi(\underline{a})\in A_{n}(\tilde \varphi,K)\right\}$ that can be shown 
as in the proof of Lemma~\ref{relation}, we obtain the desired inequality.
\end{proof}
\subsection{Samplewise level-1 upper bound}\label{sample-1}
For each $\omega\in\Omega$, define 
a Borel probability measure $\mu_{\omega,n}$ on $X$ by
\begin{equation}\label{muomegan}\mu_{\omega,n}=\frac{1}{Z_{\omega,n}}\sum_{x\in{\rm Fix}(T_{\omega}^n) }|(T_{\omega}^n)'x|^{-1}\delta_{x}\quad(n=1,2,\ldots),\end{equation}
which is a samplewise version of $\mu_n$ in \eqref{mun}.
Notice the identities
\[\frac{1}{Z_{\omega,n}}\sum_{x\in{\rm Fix}(T^n_\omega)}|(T_{\omega}^n)'x|^{-1}\tilde \varphi(\delta^n_{(\omega,x)})=\int\tilde \varphi(\delta^n_{(\omega,x)})d\mu_{\omega,n}(x)=\int\tilde \varphi d\tilde\mu^\omega_n.\]
We are going to evaluate the integral in the middle.
For a subset  $K$ of $\mathbb R$,
we denote by
$A_{\omega,n}(\tilde \varphi,K)$ the set
$\{x\in X\colon (\omega,x)\in A_{n}(\tilde \varphi,K)\},$
called {\it the $\omega$-section} 
of $A_{n}(\tilde \varphi,K)$.

\begin{lemma}\label{level-1}
For $m_p$-almost every $\omega\in\Omega$ and any closed subset $K$ of $[\inf\tilde \varphi,\sup\tilde \varphi]$,
\[\limsup_{n\to\infty}\frac{1}{n}\log \mu_{\omega,n}(A_{\omega,n}(\tilde \varphi,K))\leq -\inf_{K} I_{\Sigma,\tilde\varphi}.\]
\end{lemma}
\begin{proof}
We proceed in parallel to the proof of Proposition~\ref{ldpup-q}.
Since the interval $[\inf\tilde \varphi,\sup\tilde \varphi]$ is compact, metrizable 
and $I_{\Sigma,\tilde\varphi}$ is lower semicontinuous,
it suffices to show that
for each closed subset $K$ of this interval there exists a Borel set $\Omega_K$ with full $m_p$-measure such that for all $\omega\in\Omega_{K}$, 
\begin{equation}\label{ldpup-lem@}\limsup_{n\to\infty}\frac{1}{n}\log \mu_{\omega,n}(A_{\omega,n}(\tilde \varphi,K))\leq -\inf_{K} I_{\Sigma,\tilde\varphi}.
\end{equation}
In what follows we may assume $0<\inf_{K}I_{\Sigma,\tilde\varphi}<\infty$.
From the definitions of $\mu_n$, $\mu_{\omega,n}$ in \eqref{mun}, \eqref{muomegan} and Lemma~\ref{b-d},
for any Borel subset $B$ of $\Lambda$ we have
\[\begin{split}
\mu_n(B)&=\frac{1}{Z_{p,n}}\sum_{(\omega,x)\in {\rm Fix}(R^n)}Q_p(\omega_1\cdots\omega_n) |(T_\omega^n)'x|^{-1}\delta_{(\omega,x)}(B) \\
&=\int \sum_{x\in{\rm Fix}(T_{\omega}^n)\cap B_\omega}|(T_{\omega}^n)'x|^{-1} dm_p(\omega) \Big/ \int Z_{\omega',n}  dm_p(\omega') \\
&=\int \mu_{\omega,n}(B_\omega)\left(Z_{\omega,n} \Big/ \int_\Omega Z_{\omega',n} dm_p(\omega')\right)dm_p(\omega)\\
&\geq e^{-2c_n}\int\mu_{\omega,n}(B_\omega)\ dm_p(\omega),
\end{split}\]
where $B_\omega$ denotes the $\omega$-section of $B$.
For $\epsilon\in(0,1)$ and  $n\geq1$ we set 
\[\Omega_{\epsilon,n}=\left\{\omega\in\Omega\colon\mu_{\omega,n}(A_{\omega,n}(\tilde \varphi,K))\geq \exp\left(-n(1-\epsilon)\inf_K I_{\Sigma,\tilde\varphi}\right)\right\}.\]
Then Markov's inequality yields
\[\begin{split}
m_p(\Omega_{\epsilon,n})
&\leq \exp\left(n(1-\epsilon)\inf_K I_{\Sigma,\tilde\varphi}\right)\int\mu_{\omega,n}(A_{\omega,n}(\tilde \varphi,K))d m_p(\omega) \\
&\leq e^{2c_n}\exp\left(n(1-\epsilon)\inf_K I_{\Sigma,\tilde\varphi}\right)\mu_{n}(A_n(\tilde \varphi, K)).
\end{split}\]
By Proposition~\ref{level-1LDP}, $m_p(\Omega_{\epsilon,n})$
decays exponentially as $n$ increases. 
From Borel-Cantelli's lemma we obtain 
\eqref{ldpup-lem@} for $m_p$-almost every $\omega\in\Omega$. 
\end{proof}

\subsection{Proofs of Theorem~C and Corollary~\ref{cor-10}}\label{est2}
We put $\alpha_0=\tilde \varphi(m_p\otimes \lambda_p).$ By 
the lower semi-continuity of $I_{\Sigma,\tilde\varphi}$, the compactness of $\mathcal M(\Sigma)$ and Lemma~\ref{minimizer}, $I_{\Sigma,\tilde\varphi}(\alpha)=0$ holds if and only if $\alpha=\alpha_0$. By Lemma~\ref{level-1}, for $m_p$-almost every $\omega\in\Omega$ and any $\epsilon>0$ there exists $\alpha(\omega,\epsilon)>0$
such that for all sufficiently large $n\geq 1$ we have
\[\begin{split}\mu_{\omega,n}(A_{\omega,\epsilon})&=
\frac{1}{Z_{\omega,n}}\sum_{x\in{\rm Fix}(T^n_\omega)\cap A_{\omega,\epsilon}}|(T_{\omega}^n)'x|^{-1}\leq e^{-\alpha(\omega,\epsilon) n},\end{split}\]
where
$A_{\omega,\epsilon}=\{x\in X\colon |\tilde \varphi(\delta_{(\omega,x)}^{n} )-\alpha_0|\geq\epsilon\},$
and further
\[\begin{split}\left|\int \tilde \varphi(\delta_{(\omega,x)}^{n} )d\mu_{\omega,n}(x)-\alpha_0\right|
\leq&\frac{1}{Z_{\omega,n}}\sum_{x\in{\rm Fix}(T_\omega^n)}|(T_\omega^n)'x|^{-1}
\left|\tilde \varphi(\delta_{(\omega,x)}^{n} )-\alpha_0\right|\\
\leq&\epsilon\mu_{\omega,n}(X\setminus A_{\omega,\epsilon})+\mu_{\omega,n}(A_{\omega,\epsilon})(\sup|\tilde \varphi|+\alpha_0)\\
\leq &\epsilon+e^{-\alpha(\omega,\epsilon) n}(\sup|\tilde \varphi|+\alpha_0).
\end{split}\]
Since $\epsilon>0$ is arbitrary, we obtain
$\lim_{n\to\infty}\int \tilde \varphi(\delta_{(\omega,x)}^{n} )d\mu_{\omega,n}(x)=\alpha_0$, which completes the proof of Theorem~C.
Applying Theorem~C to acceptable functions
 $\mu\mapsto\int \varphi d\mu/\int \psi d\mu$,
 $\mu\mapsto\int \varphi d\mu\int \psi d\mu$,
 $\mu\mapsto\int g d(\mu\circ \pi_1^{-1}*\mu\circ \pi_2^{-1})$
 on $\mathcal M(\Lambda)$ completes the proof of Corollary~\ref{cor-10}. \qed

\section{Examples of application}
In this section 
we give examples to which our main results apply.
In Section~\ref{u-exp} we remark that our results apply to random dynamical systems generated by uniformly expanding Markov maps.
In Section~\ref{random-beta} we take up such a random dynamical system 
introduced in \cite{DK03}, which
generates series expansions of real numbers with non-integer bases.
In Sections~\ref{average1}, \ref{average2} and \ref{ave-mean}
we apply Theorem~C and Corollary~\ref{cor-10} to this system, and 
obtain almost-sure convergences of
time averages of digital quantities
in the expansions of random cycles. 

For random dynamical systems
generated by non-uniformly expanding Markov maps 
having neutral fixed points, the verification of (A4) is an issue.
In Section~\ref{verifya2}, we verify (A4) for maps with common (neutral) fixed points
using the thermodynamic formalism for countable Markov shifts \cite{MauUrb03}.
In Section~\ref{conv-sub-sec} we show that a weighted equidistribution of random cycles still holds even if the uniqueness of equilibrium state fails.
In Section~\ref{@} we discuss in which case (A4) holds in more interesting examples 
generated by 
maps with common 
neutral fixed points introduced in
\cite{LSV99}.
Finally in Section~\ref{extension}, we discuss some future perspectives on extensions of the results of this paper.

\subsection{Uniformly expanding maps}\label{u-exp}
 Let $T_1,\ldots,T_N$ be uniformly expanding Markov maps
 of class $C^{1+1}$ on $X$.
 Then Pelikan's condition 
 \begin{equation}\label{pelikan}\sup_{x\in X}\sum_{i=1}^N\frac{p_i}{|T_i'(x)|}<1\end{equation}
in \cite[Theorem~1]{Pel84} holds, and hence there exists a stationary measure that is absolutely continuous with respect to ${\rm Leb}$.
If (A2) holds,
such a stationary measure is unique, denoted by $\lambda_p$. If moreover (A1) holds, then (A4) follows from the thermodynamic formalism for topological Markov shifts over finite alphabet \cite{Bow75,Rue04}:
the random geometric potential $\phi$ is H\"older continuous with respect to the shift metric, and the shift-invariant measure $(m_p\otimes\lambda_p)\circ\pi$ is a Gibbs state for the potential $\phi$, and hence it is the unique equilibrium state for the potential $\phi$.
 
 In this case, (A3) is not actually necessary. Indeed, for $n\geq1$ and $\omega\in\Omega$
let $E_\omega^n$ denote the set of $x\in{\rm Fix}(T_\omega^n)$ such that $(\omega',x)\notin \pi(\Sigma)$
where $\omega'\in\Omega$ is the repetition of $\omega_1\cdots\omega_n$.
Since $\mathcal A$ is a finite set,
 $\sup_{\omega\in\Omega}\sup_{n\geq1}\#E_\omega^n/n$ is bounded from above.
By Lemma~\ref{VP} and the assumption that all the maps are uniformly expanding,
for all $\omega\in\Omega$
we have \[\lim_{n\to\infty}\frac{\sum_{x\in E_\omega^n } |(T_\omega^n)'x|^{-1}}{\sum_{x\in {\rm Fix}(T_\omega^n)} |(T_\omega^n)'x|^{-1}}=0.\]
Therefore, random cycles in $\bigcup_{\omega\in\Omega}\bigcup_{n=1}^\infty E_\omega^n$ do not affect all the previous estimates.

\subsection{Random $\beta$-expansion}\label{random-beta}
Let $\beta>1$ be a non-integer. 
A {\it greedy map} $T_1$ 
and a {\it lazy map} $T_2$
are maps from a closed interval 
 $X=[0,\lfloor \beta\rfloor/(\beta-1)]$ to itself given 
by
\[T_1(x)=
\begin{cases}
\beta x -\lfloor\beta x\rfloor & x\in[0,1), \\
\beta x -\lfloor\beta\rfloor & x\in\bigl[1,\lfloor\beta\rfloor/(\beta-1)\bigr],
\end{cases}\quad\text{and}\quad T_2(x)=u\circ T_1\circ u^{-1}(x),\]
respectively,
where $u:x\in X\mapsto \lfloor\beta\rfloor/(\beta-1)-x\in X$
and $\lfloor \cdot \rfloor$ denotes the floor function.
The restriction $T_1|_{[0,1)}$ is the 
$\beta$-transformation introduced by R\'enyi \cite{R}. See FIGURE~2 for the case $1<\beta<2$.
The skew product map $R$ in \eqref{R-def} 
generated by $T_1$ and $T_2$ is called {\it the random $\beta$-transformation} \cite{DK03}.

The random $\beta$-transformation provides series expansions of numbers in $K$. 
We define integer-valued digit functions $e_1$, $e_2$ on $X$
by 
\[e_1(x)
=\begin{cases}
\lfloor\beta x\rfloor & 
x\in\bigl[0,\frac{\lfloor\beta\rfloor}{\beta} \bigr), \\
\lfloor\beta\rfloor & 
x\in\bigl[ \frac{\lfloor\beta\rfloor}{\beta}, \frac{\lfloor\beta\rfloor}{\beta-1}\bigr],
\end{cases}\quad\text{and}\quad  
e_2(x)=\lfloor\beta\rfloor-e_1(u(x)).\]
For $(\omega,x)\in\Lambda$ and $n\geq1$ we have
\begin{equation}\label{digit-eq}T_{\omega}^{n-1}(x)=\frac{e_{\omega_n}(T_{\omega}^{n-1}(x))}{\beta}
+\frac{T_{\omega}^{n}(x)}{\beta}.\end{equation}
Using \eqref{digit-eq} recursively, we obtain a $\beta$-expansion
\begin{equation}\label{beta-exp}x=\sum_{n=1}^{\infty}\frac{e_{\omega_n}(T_{\omega}^{n-1}(x))}{\beta^n},\end{equation}
called {\it the random $\beta$-expansion of $x$ with respect to $\omega$}. 

Each number in $X$ can have an infinite number of $\beta$-expansions \cite[Theorem~1]{Er}.
The random $\beta$-transformation 
generates all possible $\beta$-expansions \cite[Theorem~2]{Da}: 
for any $\beta$-expansion of $x\in X$ there exists $\omega\in\Omega$ such that 
the random $\beta$-expansion of $x$ with respect to $\omega$ coincides with 
the given $\beta$-expansion of $x$. 
The $\beta$-expansion of $x\in X$ generated by the single map $T_1$ is called {\it the greedy expansion of $x$}.
The name {\it greedy} comes from the property that the digit sequence in the greedy expansion of a given number is the largest one,
in the lexicographical order on $\{0,1,\dots, \lfloor\beta\rfloor\}^{\mathbb N}$,
among all possible digit sequences in the $\beta$-expansion of that number \cite[Theorem 1]{Da}.
The meaning of the name {\it lazy} is analogous.
The random $\beta$-expansion with respect to $\omega=111\cdots$ coincides with the greedy $\beta$-expansion. 

\begin{figure}
\begin{center}
\includegraphics[height=4cm,width=9.2cm]{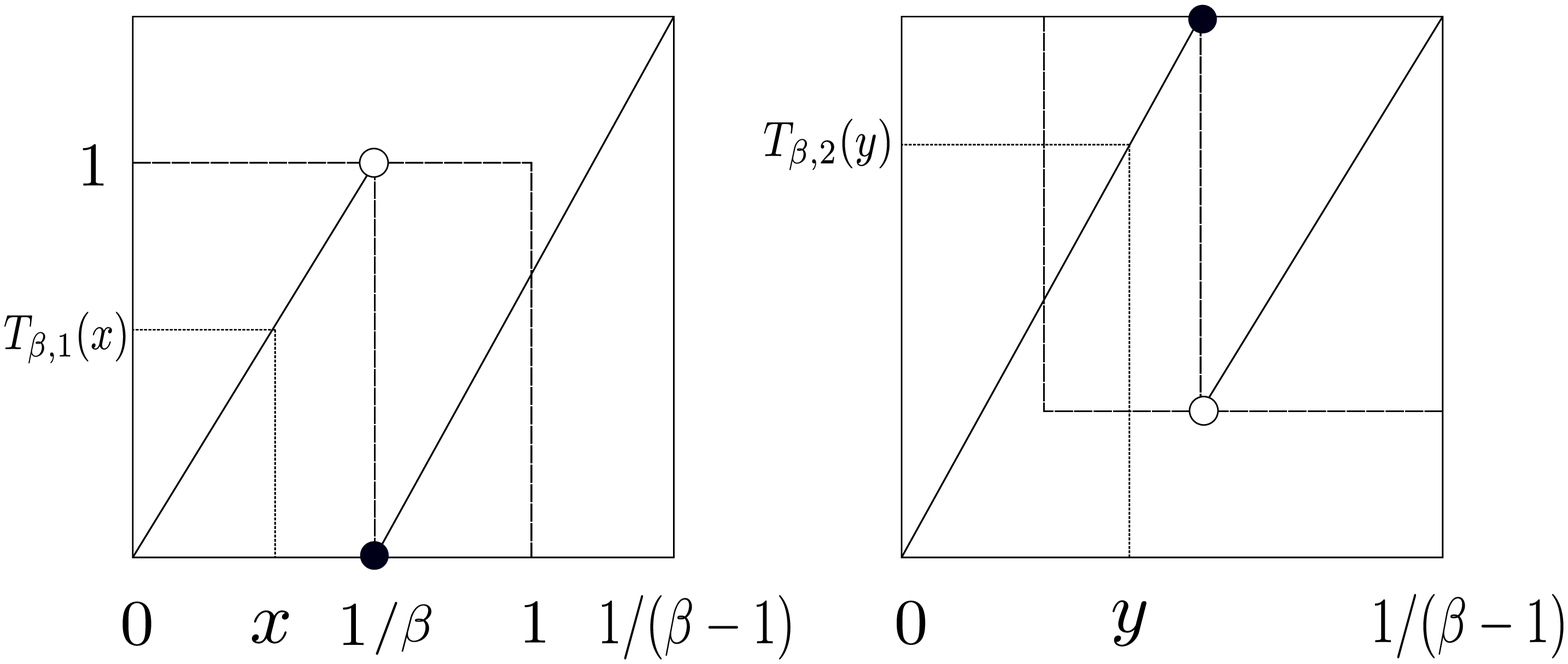}
\caption
{The greedy map $T_1$ (left) and the lazy map $T_2$ (right) for $1<\beta<2$.}
\end{center}
\end{figure}

We say {\it the greedy $\beta$-expansion of $x\in X$ is finite} if 
all but finitely many 
digits in the greedy expansion of $x$ are $0$.
From the result of Dajani and de Vries \cite[Section~4]{Da},
if the greedy $\beta$-expansion of $1$ is finite, then $T_{1}$ and $T_{2}$ are uniformly expanding Markov maps 
satisfying (A1) (A2). 
Hence, for any positive probability vector $p=(p_1,p_2)$ the random $\beta$-transformation has a unique stationary measure that is absolutely continuous with respect to the normalized
restriction of the Lebesgue measure to $X$, denoted by $\nu$.
The density of this measure, denoted by $h_{p_1}$, is explicitly given in \cite{Suz19} in terms of the random orbit $\{T_\omega^n(1)\}_{n=0}^\infty$.
For general results, see \cite{KalMag20}. 

\subsection{Average relative frequency of digits}\label{average1}
  Borel's normal number theorem states that Lebesgue almost every real number has the property
       that the limiting relative 
       frequency of each digit in the decimal expansion is $1/10$. 
       Although Borel did not use ergodic theory in his original proof, 
       this is a consequence of Birkhoff's ergodic theorem.
       An analogue of Borel's theorem for the 
       the sequence $\{e_{\omega_n}(T_\omega^{n-1}(x))\}_{n=1}^\infty$
in the random $\beta$-expansion \eqref{beta-exp} 
is a consequence of Birkhoff's ergodic theorem applied to the skew product map $R$ and its invariant probability measure $m_{p}\otimes h_{p_1}\nu$
that is ergodic \cite[Theorem~4]{Da2}. 
For $i\in\{0,1,\dots,\lfloor\beta \rfloor\}$
define 
a subinterval $L_i$ of $X$ by
\[L_i=\begin{cases}
[\frac{i}{\beta}, \frac{i+1}{\beta}) & i=0,1,\dots,\lfloor\beta\rfloor-1, \\
[\frac{\lfloor\beta\rfloor}{\beta}, \frac{\lfloor\beta\rfloor}{\beta-1}] & i=\lfloor\beta\rfloor,
\end{cases}
\]
and put
\begin{equation}\label{qi}q_i=p_1\int_{L_i}h_{p_1} d\nu+p_2\int_{L_{\lfloor\beta\rfloor-i}}h_{p_2} d\nu.\end{equation}
Note that $\sum_{i=0}^{\lfloor\beta\rfloor}q_i=1$.
By the ergodicity of $m_{p}\otimes h_{p_1}\nu$, 
for $m_p$-almost every $\omega=(\omega_k)_{k=1}^\infty\in\Omega$
the relative frequency with which the integer $i$ appears in the random $\beta$-expansion of $x\in X$,
namely the number
\[F_{\omega,n}(i,x)=\frac{1}{n}\#\{1\leq k\leq n\colon e_{\omega_k}(T_{\omega}^{k-1}(x))=i\},\]
converges to $q_i$ as $n\to\infty$ for Lebesgue almost every $x\in X$. 
On relative frequencies of digits in $\beta$-expansions of random cycles,
using Theorem~C and Corollary~\ref{cor-10}(a)
we obtain the following result.

\begin{prop}[Almost-sure convergence of the average relative frequency of digits]\label{thm-beta}
Let $\beta>1$ be a non-integer such that the greedy $\beta$-expansion of $1$ is finite, and let
$p=(p_1,p_2)$ be a positive probability vector.
For $m_p$-almost every sample $\omega\in\Omega$ and all $i=0,1,\dots, \lfloor\beta\rfloor$ we have
\[\label{eq-cor}\begin{split}
    \lim_{n\to\infty}\frac{1}{\#{\rm Fix}(T_{\omega}^n)} \sum_{x\in{\rm Fix} (T_{\omega}^n)}&F_{\omega,n}(i,x)
    =q_i.
\end{split}\]

\end{prop}

 \begin{proof}
 Since 
  $|T_{1}'|=|T_{2}'|=\beta$ we have 
 $Z_{\omega,n}
 =\#{\rm Fix}(T^n_\omega)/\beta^n.$ Define $\psi_i\colon\Lambda\to\{0,1\}$ by
 \[\psi_i(\omega,x)=\begin{cases}
\1_{L_i}(x), & \omega_1=1, \\
\1_{u(L_{\lfloor \beta\rfloor-i})}(x), & \omega_1=2,
 \end{cases}\]
 where $\1_A$ denotes the indicator function of a set $A\subset\mathbb R$.
 Since $e_1(x)=i$ for $x\in L_i$ and $e_2(x)=i$ for $x\in u(L_{\lfloor\beta\rfloor-i})$ we have
 \[\frac{1}{n}\sum_{k=0}^{n-1}\psi_i(R^k(\omega,x))=
 F_{\omega,n}(i,x).\]
    By Corollary~\ref{cor-10}(a), 
 the limit in question exists and is equal to $\int \psi_id(m_{p}\otimes h_{p_1}\nu)$. We have
    \[\begin{split}
     \int \psi_i d(m_{p}\otimes h_{p_1}\nu)
     &=p_1\int\1_{L_i}h_{p_1}d\nu+p_2\int\1_{u(L_{\lfloor\beta\rfloor-i})}h_{p_1}
     d\nu\\
     &=p_1\int_{L_i}h_{p_1}d\nu
     +p_2\int_{u(L_{\lfloor\beta\rfloor-i})}h_{p_2}\circ u d\nu=q_i.
 \end{split}\]
 The second equality follows from
 $h_{p_1}=h_{p_2}\circ u$ by
 the relation $T_{2}=u\circ T_{1}\circ u^{-1}$.
\end{proof}
 
 In some particular cases we can compute the exact value of $q_i$ and hence the limit 
 in Proposition~\ref{thm-beta}.
 For $\beta=(1+\sqrt{5})/2$,
the greedy expansion of $1$ is finite
since $1=1/\beta+1/\beta^2$.
 By \cite[Theorem~4.4]{Suz19} 
 we have 
\[h_{p_1}=\frac{\beta}{3-\beta}\Bigl((1-p_1)\beta {\1}_{[0,1/\beta]}+{\1}_{[1/\beta,1]}+p_1 \beta {\1}_{[1,1/(\beta-1)]}\Bigr).\]
Substituting this into the right-hand side of \eqref{qi} 
 shows that 
\[q_i=
\frac{1}{2}\left(1+\frac{2p_{i+1}-1}{\sqrt{5}}\right)\ \text{ for }i=0,1.\]

\subsection{Average symmetric mean of digits}\label{average2}
The {\it (second) symmetric mean} of a finite sequence $\{a_k\}_{k=1}^n$ of real numbers
is the quantity
\[S(a_1,\ldots,a_n)=\frac{2}{n(n-1)}\sum_{1\leq i<j\leq n}a_i a_j.\]
Symmetric means of digits in expansions of numbers are interesting quantities to look at. 
Their limiting behaviors 
for the regular continued fraction expansion were investigated by Cellarosi et al. \cite{CHSJ15}. 

We consider symmetric means of digits in
the random $\beta$-expansion.
Set $\pi_e(\omega,x)=e_{\omega_1}(x)$ for $(\omega,x)\in\Lambda$.
By Birkhoff's ergodic theorem,
  for $m_{p}\otimes h_{p_1}\nu$-almost all $(\omega,x)\in\Lambda$
   we have
\[\begin{split}
\lim_{n\to\infty}
S(e_{\omega_1}(x),\ldots, e_{\omega_n}(x))
&=\lim_{n\to\infty}\frac{1}{n(n-1)}\left(\sum_{k=1}^n e_{\omega_k}(x)\right)^2 \\
&=\Bigl(\int_{\Omega\times K}\pi_e(\omega, x)d(m_{p}\otimes h_{p_1}\nu)\Bigr)^2  =\left(\sum_{i=0}^{\lfloor\beta\rfloor}iq_i\right)^2.\end{split}\]
This convergence does not imply
the convergence of the average of symmetric means of digits in $\beta$-expansions of random cycles. Applying Theorem~C and Corollary~\ref{cor-10}(b) we obtain the following result.


\begin{prop}[Almost-sure convergence of the average symmetric mean of digits]\label{thm-beta2}
Let $\beta$ and 
$p$ be as in Proposition~\ref{thm-beta}.
For $m_p$-almost every sample $\omega=(\omega_k)_{k=1}^\infty\in\Omega$ we have
\[\label{eq-cor}\begin{split}
    \lim_{n\to\infty}\frac{1}{\#{\rm Fix}(T_{\omega}^n)} \sum_{x\in{\rm Fix} (T_{\omega}^n)}&
    S(e_{\omega_1}(x),\ldots, e_{\omega_n}(x))=
    \Bigl(\sum_{i=0}^{\lfloor\beta\rfloor}iq_i\Bigr)^2.
    \end{split} \]
\end{prop}
\begin{proof}
Notice the identity
\[S(e_{\omega_1}(x),\ldots, e_{\omega_n}(x))=\frac{1}{n(n-1)}
\left(\left(\sum_{k=1}^n e_k(x)\right)^2-\sum_{k=1}^n e_k^2(x)\right).\]
Therefore, applying
 Corollary~\ref{cor-10}(b) with $\varphi=\psi=\pi_e$ we obtain the desired equality.
Since the digits are uniformly bounded, the contribution from the second series in the parenthesis is diminished in the limit $n\to\infty$.
\end{proof}

\subsection{Average of mean distances of digits}\label{ave-mean}
The {\it mean distance} of a finite sequence $\{a_k\}_{k=1}^n$ of real numbers is given by 
\[D(a_1,\ldots,a_n)=\frac{2}{n(n-1)}\sum_{1\leq i<j\leq n}|a_i-a_j|.\]
We obtain the following result as an application of Corollary\ref{cor-10}(c). 


\begin{prop}
[Almost-sure convergence of the average mean distances of digits]
Let $\beta$ and 
$p$ be as in Proposition~\ref{thm-beta}.
For $m_p$-almost every sample $\omega=(\omega_k)_{k=1}^\infty\in\Omega$ we have
\[\label{eq-cor}
    \lim_{n\to\infty}\frac{1}{\#{\rm Fix}(T_{\omega}^n)} \sum_{x\in{\rm Fix} (T_{\omega}^n)}
    D(e_{\omega_1}(x),\ldots, e_{\omega_n}(x))=
    2\sum_{0\leq i<j\leq\lfloor\beta\rfloor}(j-i)q_iq_j.\]
\end{prop}

\begin{proof}
A key ingredient is the next representation of
measures in terms of 
the relative frequency of digits. 

\begin{lemma}\label{thm-beta3}
The following hold:
\begin{itemize}
\item[(a)] $(m_p\otimes h_{p_1}\nu)\circ\pi_e^{-1}=\sum_{i=0}^{[\beta]}q_i\delta_i;$
\item[(b)]$(m_p\otimes h_{p_1}\nu)\circ(-\pi_e^{-1})=\sum_{i=0}^{[\beta]}q_i\delta_{-i};$

\item[(c)] 
$(m_p\otimes h_{p_1}\nu)\circ\pi_e^{-1}*(m_p\otimes h_{p_1}\nu)\circ(-\pi_e)^{-1}=\sum_{i=0}^{\lfloor\beta\rfloor}\sum_{j=0}^{\lfloor\beta\rfloor}q_iq_j\delta_{i-j}$.
\end{itemize}
\end{lemma}

\begin{proof}
For any $i\in\{0,1,\dots,\lfloor\beta \rfloor\}$,
the set $\{(\omega, x)\in\Lambda\colon \pi_e(\omega,x)=i\}$ is the disjoint union of
$\{\omega\in\Omega\colon \omega_1=1\}\times L_i$ and $\{\omega\in\Omega\colon  \omega_1=2\}\times u(L_{\lfloor \beta\rfloor-i})$. Then we have
$(m_p\otimes h_{p_1}\nu)\circ\pi^{-1}_e(\{i\})
=q_i$, and so
for any Borel subset $B$ of $\mathbb R$,
\[(m_p\otimes h_{p_1}\nu)\circ\pi_e^{-1}(B)=\sum_{i=0}^{[\beta]}q_i\delta_i(B).\]
Hence (a) holds. A proof of (b) is analogous.
A direct calculation gives
\[
\left(\sum_{i=0}^{[\beta]}q_i\delta_i\right)*\left(\sum_{j=0}^{[\beta]}q_j\delta_{-j}\right)(B)=\sum_{i=0}^{\lfloor\beta\rfloor}q_i\sum_{j=0}^{\lfloor\beta\rfloor}q_j\delta_{-j}(B-\{i\}),\]
where $B-\{i\}=\{x-i\colon  x\in B\}$.
Since $\delta_{-j}(B-\{i\})=\delta_{i-j}(B)$
we obtain (c).
\end{proof}

By Lemma~\ref{thm-beta3}(c) and Corollary~\ref{cor-10}(c) applied to the functions $\pi_1=\pi_e$, $\pi_2=-\pi_{e}$ and a bounded continuous function $g\colon\mathbb R\to\mathbb R$ such that $g(x)=|x|$ for all $x\in[-\lfloor\beta\rfloor,\lfloor\beta\rfloor]$,
for $m_p$-almost every $\omega\in\Omega$ we have 
\[\begin{split}
    \lim_{n\to\infty}\frac{1}{\#{\rm Fix}(T_{\omega}^n)}& \sum_{x\in{\rm Fix} (T_{\omega}^n)}
    D(e_{\omega_1}(x),\ldots, e_{\omega_n}(x)) \\
    &=\lim_{n\to\infty}\frac{1}{\#{\rm Fix}(T_{\omega}^n)} \sum_{x\in{\rm Fix} (T_{\omega}^n)}\frac{2}{n(n-1)}\sum_{1\leq i<j\leq n}|e_{\omega_i}(x)-e_{\omega_j}(x)| \\
    &=\lim_{n\to\infty}\frac{1}{\#{\rm Fix}(T_{\omega}^n)} \sum_{x\in{\rm Fix}(T_{\omega}^n)}\frac{1}{n^2}
    \sum_{i=1}^n\sum_{j=1}^n|e_{\omega_i}(x)-e_{\omega_j}(x)| \\
    &=\int_{\mathbb{R}}|x|d\left(\sum_{i=0}^{\lfloor\beta\rfloor}\sum_{j=0}^{\lfloor\beta\rfloor}q_iq_j\delta_{i-j}\right)\\
    &=\sum_{i=0}^{\lfloor\beta\rfloor}\sum_{j=0}^{\lfloor\beta\rfloor}|i-j|q_iq_j 
    =2\sum_{0\leq i<j\leq\lfloor\beta\rfloor}(j-i)q_iq_j,
\end{split}\]
as required. 
\end{proof}

 \subsection{ Non-uniformly expanding maps with common fixed points}\label{verifya2}
 Let $T$ be a Markov map on $[0,1]$.
We say $x\in T$ is {\it a neutral fixed point} of $T$
if $T(x)=x$ and $|T'(x)|=1$.
Deterministic dynamical systems generated by iterations of Markov maps having neutral fixed points have been well investigated, see  \cite{PomMan80,Sch75,Tha83} for example.

  Under the notation in the beginning of Section~\ref{markov},
let 
$\mathcal A(1)=\{1,3\}$, $\mathcal A(2)=\{2,4\}$ and
let $T_1,T_2$ be non-uniformly expanding Markov maps 
of class $C^{1+\tau}$, $\tau>0$ on $X=[0,1]$
with Markov partitions $\{J(a)\}_{a\in\mathcal A(1)}$, $\{J(a)\}_{a\in\mathcal A(2)}$ respectively
such that the following hold:

\begin{itemize}

\item[(B1)]
for any $a\in\mathcal A=\{1,2,3,4\}$,
$T_{i(a)}(J(a))\supset(0,1);$

\item[(B2)] 
$T_1(0)=T_2(0)=0$, and
$0$ is a neutral fixed point of $T_1$;

\item[(B3)] 
there exists $\gamma>1$ such that
$\inf|(T_{1}|_{J(3)})'|\geq\gamma$ and 
$\inf|(T_{2}|_{J(4)})'|\geq\gamma$.

\end{itemize}
 FIGURE~1 gives a schematic picture of the partition of the space $\Lambda=\Omega\times X$.
 Condition (B1) means that $T_1$, $T_2$ are fully branched, which implies (A1) (A2) (A3).
The issue is (A4).
In order to control distortions of random compositions on $J(1)\cap J(2)$, we additionally assume
\begin{itemize}
\item[(B4)]
there exists $\epsilon>0$ such that for $a=1,2$,
$T_{a}|_{J(a)\setminus\{0\}}$
 can be extended to a $C^3$ map on 
 an interval of length $|J(a)|+\epsilon$
 with negative Schwarzian derivative.
\end{itemize}
Recall that a real-valued $C^3$ function $f$ on an interval
 has {\it negative Schwarzian derivative} if $f'''/f'-(3/2)(f''/f')^2<0$.
It is well-known that $T_1$, $T_2$ have a sigma-finite invariant measure
that is absolutely continuous with respect to ${\rm Leb}$, abbreviated as an {\it acim}.
For $i=1,2$ define $t_i\colon J(i)\to\mathbb N\cup\{\infty\}$ by
\[t_i(x)=\inf\{n\geq1\colon T_i^n(x)\in J(i)\}.\]

\begin{figure}
\begin{center}
\includegraphics[height=5.5cm,width=6cm]{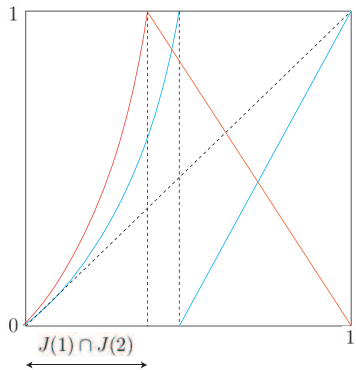}
\caption
{The maps $T_1$ (red) and $T_2$ (blue) in Proposition~\ref{ex2-prop}(b)(c) having $0$ as a common neutral fixed point.}
\end{center}
\end{figure}

\begin{prop}\label{ex2-prop}
For $T_1$, $T_2$ as above the following hold:
\begin{itemize}
    \item[(a)] if
$0$ is not a neutral fixed point of $T_2$, then
for any positive probability vector $p=(p_1,p_2)$,
there exists a stationary measure $\lambda_p$ that is absolutely continuous with respect to ${\rm Leb}$ such that $(m_p\otimes\lambda_p)\circ\pi$ is the unique equilibrium state for the random geometric potential $\phi$. In particular, (A4) holds.
\item[(b)] if
 $0$ is a neutral fixed point of $T_2$,
 $T_1(x)\geq T_2(x)$ for all $x\in J(1)\cap J(2)$
and $\int_{J(1)} t_1(x)dx=+\infty$, then
for
 any positive probability vector $p=(p_1,p_2)$,
 $(m_p\otimes\delta_0)\circ\pi$ is the unique equilibrium state for the potential $\phi$. 
 In particular, (A4) holds.
 
 \item[(c)] if
 $0$ is a neutral fixed point of $T_2$, 
 $T_1(x)\geq T_2(x)$ for all $x\in J(1)\cap J(2)$ and $\int_{J(2)} t_2(x) dx<+\infty$, then
for
 any positive probability vector $p=(p_1,p_2)$,
 there exists a stationary measure $\lambda_p$ that is absolutely continuous with respect to ${\rm Leb}$ such that 
  equilibrium states for the potential $\phi$
 are precisely the convex combinations
 of $(m_p\otimes\lambda_p)\circ\pi$ and $(m_p\otimes\delta_0)\circ\pi$.
 In particular, (A4) does not hold.
 \end{itemize}
\end{prop}

\begin{remark}\label{rem-acim}
{\rm It is well-known that $\int_{J(i)} t_i(x)dx$ is finite if and only if  acims of $T_i$ can be normalized.
In case (a) of Proposition~\ref{ex2-prop}, by (B4) and the minimum principle \cite[Chapter II, Lemma~6.1]{dMevSt93},
$T_2$ is uniformly expanding, or else $\lim_{x\nearrow x_0}T_2'(x_0)=1$
where $x_0$ denotes the endpoint of $J(2)$ other than $0$.
 In particular, $T_2^2$ is uniformly expanding. Hence,
 acims of $T_2$ can be normalized. In case (b), acims of both maps
cannot be normalized (they are infinite measures). Since $0$ is a common fixed point of $T_1$ and $T_2$, $\delta_0$ is a stationary measure.
This case applies to random dynamical systems generated by some L-S-V maps \cite{LSV99}, to be discussed in Section~\ref{@}.
In case (c), acims of both maps can be normalized.}
\end{remark}


   \begin{proof}[Proof of Proposition~\ref{ex2-prop}]
The difficulty is that the random geometric potential $\phi$
on the space $\Sigma=\mathcal A^{\mathbb N}=\{1,2,3,4\}^\mathbb N$ is not H\"older continuous.
  We construct a full shift over an infinite alphabet using the first return map to the subset $[3]\cup[4]$ of $\Sigma$, and
 then appeal to the results on the existence and uniqueness of equilibrium states for countable Markov shifts \cite{MauUrb03}.

\begin{figure}
\begin{center}
\includegraphics[height=4cm,width=7cm]{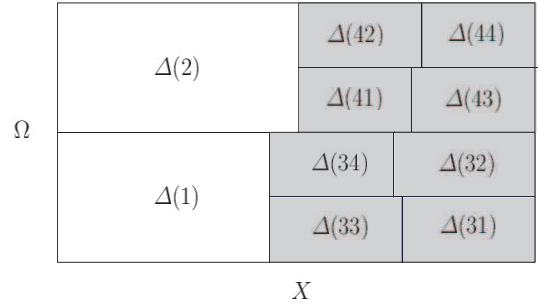}
\caption
{
The projection of the inducing domain $[3]\cup[4]$ to $\Lambda$ (the shaded area).}
\end{center}
\end{figure}

 Define a function  \[t\colon\underline{a}\in \Sigma\mapsto\inf\{n\geq1\colon \sigma^{n}\underline{a}\in [3]\cup[4]
\}\in\mathbb N\cup\{+\infty\},\]
which is the first entry time to $[3]\cup[4]$.
For each $n\in\mathbb N\cup\{+\infty\}$, write
$\{t=n\}$ for the set $\{\underline{a}\in[3]\cup[4]\colon t(\underline{a})=n\}$.
Define an induced map \[\hat\sigma\colon \underline{a}\in([3]\cup[4])\setminus\{t=+\infty\}\mapsto \sigma^{t(\underline{a})}\underline{a}\in[3]\cup[4],\]
 and set 
\[\hat\Sigma=\bigcap_{n=0}^\infty \hat\sigma^{-n}(([3]\cup[4])\setminus\{t=+\infty\}).\]
Since $\hat\sigma$ is surjective,
$\hat\sigma(\hat\Sigma)=\hat\Sigma$ holds.
We introduce an empty word $\emptyset$, set $\{1,2\}^0=\{\emptyset\}$ 
and define $a\emptyset=a=\emptyset a$, $a\emptyset b=ab$ for $a,b\in\mathcal A$.
For each $a\in\{3,4\}$, $j\geq0$ and $\bold i^j\in\{1,2\}^j$,
$\hat\sigma$ maps set
$[a \bold i^{j}3]\cup[a \bold i^{j}4]$
bijectively onto $[3]\cup[4]$.
The set $([3]\cup[4])\setminus\{t=+\infty\}$ is partitioned into sets of this form.
We introduce an infinite discrete topological space
\[\hat{\mathcal A}=\left\{[a \bold i^j3]\cup[a \bold i^j4]\colon a\in\{3,4\},j\geq0,\bold i^j\in\{1,2\}^j\right\},\]
and an associated one-sided Cartesian product topological space
 \[{\hat{\mathcal A}}^{\mathbb N}=\{\underline{\hat a}=(\hat a_n)_{n=1}^\infty\colon \hat a_n\in\hat{\mathcal A}\text{ for all }n\geq1\}.\] 
Define a map $\iota\colon
\bigcup_{n=1}^\infty{\hat{\mathcal A}}^{n}\to\bigcup_{n=1}^\infty\mathcal A^n$ by 
\[\iota(\hat a_1\hat a_2\cdots \hat a_n)=a_1\bold i^{j_1}a_2\bold i^{j_2}\cdots a_n\bold i^{j_n},\]
   where
    $\hat a_k=[a_k \bold i^{j_k}3]\cup[a_k \bold i^{j_k}4]\in\hat{\mathcal A}$ for $1\leq k\leq n$.
For $n\geq1$ and  $\hat a_1\cdots \hat a_n\in{\hat{\mathcal A}}^n$, we set
  \[[\![\hat a_1\cdots \hat a_n]\!]=[\iota(\hat a_1\hat a_2\cdots \hat a_n)3 ]\cup[\iota(\hat a_1\hat a_2\cdots \hat a_n)4 ],\]
\[\hat\varDelta(\hat a_1\cdots \hat a_n)=\varDelta(\iota(\hat a_1\cdots \hat a_n))\quad\text{and}\quad F_{\hat a_1\cdots \hat a_n}=f_{\iota(\hat a_1\cdots \hat a_n)}.\]
Recall \eqref{remind1} and \eqref{remind2}.

The map $\iota_\infty\colon{\hat{\mathcal A}}^{\mathbb N}\to\hat\Sigma$
   given by \[\iota_\infty(\underline{\hat a})\in\bigcap_{n=1}^\infty[\![\hat a_1\cdots \hat a_n]\!]\] is a homeomorphism commuting with the left shifts on ${\hat{\mathcal A}}^\mathbb N$ and $\hat\Sigma$. For ease of notation, we will sometimes identify $\underline{\hat a}$ with $\iota_\infty(\underline{\hat a})$.
If
$\iota_\infty(\underline{\hat a})=a_1\bold i^{j_1}a_2\bold i^{j_2}\cdots a_n\bold i^{j_n}\cdots$ then
$t(\underline{\hat a})=j_1+1$.
We introduce an induced potential \[\Phi\colon \hat{\underline{a}}\in \hat\Sigma\mapsto
 \sum_{k=0}^{t(\underline{\hat a})-1}\phi(\sigma^k\hat{\underline{a}}),\]
and an associated induced pressure 
\[P(\Phi)=\lim_{n\to\infty}\frac{1}{n}\log\sum_{\hat a_1\cdots \hat a_n\in {\hat{\mathcal A}}^n}
\sup_{[\![\hat a_1\cdots \hat a_n]\!]}\exp\left(\sum_{k=0}^{n-1}\Phi\circ\hat\sigma^k\right).\]
By \cite[Theorem~2.1.8]{MauUrb03}, the variational principle holds:
\begin{equation}\label{VarP}P(\Phi)=\sup\left\{ h(\hat\nu)+\int
\Phi  d\hat\nu\colon\hat\nu\in\mathcal M(\hat\Sigma,\hat\sigma|_{\hat\Sigma}), \int\Phi d\hat\nu>-\infty\right\},\end{equation}
where $\mathcal M(\hat\Sigma,\hat\sigma|_{\hat\Sigma})$
denotes the set of $\hat\sigma|_{\hat\Sigma}$-invariant Borel probability measures on $\Sigma$ whose supports are contained in $\hat\Sigma$ and
$ h(\hat \nu)$ denotes the measure-theoretic entropy of $\hat\nu$ 
relative to $\hat{\sigma}|_{\hat\Sigma}$.
Measures which attain the supremum in \eqref{VarP} are called {\it equilibrium states for the induced potential $\Phi$}.

In case (b) or (c), there is a ``trivial'' equilibrium state for the random geometric potential.
\begin{lemma}\label{neutralT2}
If $|T_2'(0)|=1$, 
 $(m_p\otimes\delta_0)\circ\pi$ is an equilibrium state for the potential $\phi$.
 \end{lemma}
\begin{proof} 
The measure $(m_p\otimes\delta_0)\circ\pi$
is the only measure in $\mathcal M(\Sigma,\sigma)$ which does not give positive weight to $\hat\Sigma$.
 Since $|T_1'(0)|=1=|T_2'(0)|$ we have
  \[\int \phi d((m_p\otimes\delta_0)\circ\pi)=p_1\log p_1+p_2\log p_2.\]
 Since $(\sigma|_{\Omega\times\{0\}},(m_p\otimes\delta_0)\circ\pi)$ is isomorphic to the one-sided $(p_1,p_2)$-Bernoulli shift,
 for the entropy we have
 \[h((m_p\otimes\delta_0)\circ\pi)=-p_1\log p_1-p_2\log p_2.\]
 Hence
 $F((m_p\otimes\delta_0)\circ\pi)=0.$ Since
 $P(\phi)=0$ by Lemma~\ref{exist-equi},
 $(m_p\otimes\delta_0)\circ\pi$ is an equilibrium state for the potential $\phi$.
\end{proof}

In order to construct a non-trivial equilibrium state,
we aim to verify sufficient conditions in
\cite[Theorem~2.2.9]{MauUrb03} and
\cite[Corollary~2.7.5]{MauUrb03} for the existence and uniqueness of a shift-invariant Gibbs-equilibrium state for a countable Markov shift.
We fix a metric $d$ that generates the topology on ${\hat{\mathcal A}}^\mathbb N$ 
by setting $d(\underline{\hat a},\underline{\hat b})=\exp(-\inf\{n\geq1\colon \hat a_n\neq \hat b_n\})$
where $\exp(-\infty)=0$ by convention.
\begin{lemma}\label{holder}
There exist $C>0$ and $\tau_0>0$ such that
for all $\hat a\in\hat{\mathcal A}$ and all $\hat{\underline{b}}$, $\hat{\underline{c}}\in[\![\hat a]\!]$ we have
$|\Phi(\hat{\underline{b}})-\Phi(\hat{\underline{c}})|\leq C\cdot d(\hat{\underline{b}},\hat{\underline{c}})^{\tau_0}.$
\end{lemma}


\begin{proof}

From (B4) and the bounded distortion result \cite[Chapter~IV, Theorem~1.2]{dMevSt93}
 based on Koebe's principle,
there exists $C>0$ such that for all
$n\geq1$, $\bold i^n\in\{1,2\}^n$, $a\in\{3,4\}$ and 
$x,y\in J(\bold i^na)$ we have
\[D(f_{\bold i^n},x,y)\leq C
|f_{\bold i^n}(x)-f_{\bold i^n}(y)|.\]
Hence, there exists $C>0$ such that
 for all $\hat{\underline{b}}$, $\hat{\underline{c}}\in[\![\hat a]\!]$ with $\hat{\underline{b}}\neq\hat{\underline{c}}$ we have
\begin{equation}\label{holder-eq1}\Phi(\hat{\underline{b}})-\Phi(\hat{\underline{c}})\leq C|F_{\hat a}(\Pi(\pi(\hat{\underline{b}})))-F_{\hat a}(\Pi(\pi(\hat{\underline{c}})))|,\end{equation}
see \eqref{diagram}.
We have $\hat a=\hat b_1=\hat c_1$,
and there exists an integer $k\geq2$ such that
$d(\hat{\underline{b}},\hat{\underline{c}})=e^{-k}$. 
By (B2) and the mean value theorem,
in the case $k\geq3$ we have
\begin{equation}\begin{split}\label{holder-eq2}|F_{\hat b_1}(\Pi(\pi(\hat{\underline{b}})))-F_{\hat b_1}(\Pi(\pi(\hat{\underline{c}})))|&\leq 
\frac{|F_{\hat b_1\cdots \hat b_{k-1}}
(\Pi(\pi(\hat{\underline{b}})))-F_{\hat b_1\cdots \hat b_{k-1}}
(\Pi(\pi(\hat{\underline{c}})))|}{ \inf|F_{\hat b_2\cdots \hat b_{k-1}}'|}\\
&\leq \gamma^{-k+2}.\end{split}\end{equation}
Put $\tau_0=\log\gamma$. 
Combining \eqref{holder-eq1} and \eqref{holder-eq2}
we obtain the desired inequality. The case $k=2$ is covered by \eqref{holder-eq1}.
\end{proof}

\begin{lemma}\label{zero-p}
We have 
$P(\Phi)=0$ and
$\sum_{\hat a\in\hat{\mathcal A}}\sup_{[\![\hat a]\!]}\exp\Phi<+\infty$.
\end{lemma}
\begin{proof}
From Lemma~\ref{holder}, there exists $C>0$ such that for all $n\geq1$ and $\hat a_1\cdots \hat a_n\in \hat{\mathcal A}^{n}$ we have
\[\sup_{\underline{b},\underline{c}\in[\![\hat a_1\cdots \hat a_n]\!]}\left( \sum_{k=0}^{n-1}\Phi(\hat\sigma^k\hat{\underline{b}})-\sum_{k=0}^{n-1}\Phi(\hat\sigma^k
\hat{\underline{c}})\right)\leq C\sum_{k=0}^{n-1}e^{-k\tau_0}\leq \frac{C}{e^{\tau_0}-1}.\]
This implies
\[(m_p\otimes{\rm Leb})\left(\hat\varDelta(\hat a_1\cdots \hat a_n)\right)\asymp
\sup_{[\![\hat a_1\cdots \hat a_n]\!]} 
\exp\left(\sum_{k=0}^{n-1}\Phi\circ\hat\sigma^k\right),\]
where $\asymp$ indicates that there exists a constant $C>1$ such that the ratio of the two numbers are bounded below by $C^{-1}$ and above by $C$ for any $n\geq1$.
Since
$\sum_{\hat a_1\cdots \hat a_n\in \hat{\mathcal A}^n}(m_p\otimes{\rm Leb})(\hat\varDelta(\hat a_1\cdots \hat a_n))=1-p_1|J(1)|-p_2|J(2)|>0$,
rearranging the double inequalities,
summing the results over all $\hat a_1\cdots \hat a_n\in \hat {\mathcal A}^n$, and then taking logarithms, dividing by $n$ and letting $n\to\infty$ yields $P(\Phi)=0$. 
The second claim follows from combining these estimates with $n=1$.
\end{proof}

By \cite[Corollary~2.7.5]{MauUrb03}
together with Lemmas~\ref{holder} and \ref{zero-p},
there exists a unique measure $\hat\mu_p\in\mathcal M(\hat\Sigma,\hat\sigma|_{\hat\Sigma})$ with 
the Gibbs property, namely,
for $n\geq1$ and $\hat a_1\cdots \hat a_n\in {\hat{\mathcal A}}^n$ we have
\begin{equation}\label{gibbs}
\hat\mu_p[\![\hat a_1\cdots \hat a_n]\!]\asymp
\sup_{[\![\hat a_1\cdots \hat a_n]\!]}\exp\sum_{k=0}^{n-1}\Phi\circ\hat\sigma^k.\end{equation}


\begin{lemma}\label{finite} In case (a) or (c) of Proposition~\ref{ex2-prop},
$\int t d\hat\mu_p<+\infty$ and $\int\Phi d\hat\mu_p<+\infty$.
In case (b) of Proposition~\ref{ex2-prop},
$ \int t d\hat\mu_p=+\infty.$\end{lemma}
\begin{proof}Since $t$ is constant on each $1$-cylinder $[\![\hat a]\!]$, $\hat a\in \hat{\mathcal A}$ we denote this constant value by 
$t(\hat a)$. For all $\hat a\in\hat{\mathcal A}$ we have
 \[\sup_{[\![\hat a]\!]}|\Phi|\leq
t(\hat a)\left(\max_{1\leq i\leq 2}|\log p_i|+\max_{1\leq i\leq 2}\sup\log|T_i'|\right).\]
Hence, the finiteness of
$\int\Phi d\hat\mu_p$ follows from that of
 $\int td\hat\mu_p$.

In case (a), $0$ is not a neutral fixed point of $T_2$ as in Remark~\ref{rem-acim}. Hence,
there exists $\rho>1$ such that
$|T_2'(x)|\geq \rho$
for all $x\in J(21)\cup J(22)$.
Put
\[\zeta=p_1+p_2\rho^{-1}\in(0,1).\] From \eqref{gibbs}, there exists $C>0$ such that for $a\in\{3,4\}$ we have
\[\sum_{\stackrel{\hat a\in \hat{\mathcal A}}{t(\hat a)=n,[\![\hat a]\!]\subset[a]}} \hat\mu_p[\![\hat a]\!]\leq
C\left(\sup_{[11]\cup[12]\cup[21]\cup[22]}\exp\phi\right)^n\leq C\zeta^n.\]
Since $t\equiv1$ on $[33]\cup[34]\cup[43]\cup[44]$, we obtain
\[\begin{split}\int t d\hat\mu_p&=\hat\mu_p([33]\cup[34]\cup[43]\cup[44])+
\sum_{n=2}^\infty n\sum_{a=3,4}\sum_{\stackrel{\hat a\in \hat{\mathcal A}}{t(\hat a)=n,[\![\hat a]\!]\subset[a]}} \hat\mu_p[\![\hat a]\!]\\
&\leq1+2C\sum_{n=2}^\infty  n\zeta^n<+\infty,
\end{split}\]
as required.

Since $\hat\mu_p$ has no atom, we may exclude from further consideration all those points in $\hat\Sigma$ at which $\pi$ is not one-to-one.
In case (b) or (c), the common assumption
 $T_1(x)\geq T_2(x)$ for all $x\in J(1)\cap J(2)$ 
implies that
there exists $C>0$ such that 
for $a=3,4$, $i=1,2$ and for all $(\omega,x)\in \varDelta(ai)$ we have
\begin{equation}\label{t1t2}t_1(T_i(x))-C\leq t(\pi^{-1}(\omega,x))\leq
t_2(T_i(x))+C.\end{equation}
By the first inequality in \eqref{t1t2}, 
for all $\omega\in\Omega$ such that $\omega_1=a-2$ we have
\begin{equation}\label{infty-eq1} \int_{\varDelta(ai)_\omega} t(\pi^{-1}(\omega,x))
dx \geq\int_{\varDelta(ai)_\omega} t_1(T_i(x))dx-C.\end{equation}
Recall that $\varDelta(ai)_\omega$ denotes the $\omega$-section of $\varDelta(ai)$.
By the second inequality in \eqref{t1t2}, 
for all $\omega\in\Omega$ such that $\omega_1=a-2$ 
we have
\begin{equation}\label{infty-eq10} \int_{\varDelta(ai)_\omega} t(\pi^{-1}(\omega,x))
dx\leq\int_{\varDelta(ai)_\omega} t_2(T_i(x))dx+C.\end{equation} 
By Fubini's theorem, for $a=3,4$ and $i=1,2$ we have
\begin{equation}\label{infty-eq20}\begin{split}
\int_{[ai]}  t d(m_p\otimes{\rm Leb})\circ\pi&=\int_{\varDelta(ai)}  t(\pi^{-1}(\omega,x))d(m_p\otimes{\rm Leb})\\
&=\int dm_p(\omega)\int_{\varDelta(ai)_\omega} t(\pi^{-1}(\omega,x))
dx.\end{split}\end{equation}
Since the density $d\hat\mu_p/d(m_p\otimes{\rm Leb})\circ\pi$
is uniformly bounded away from zero and infinity almost everywhere,
in case (b) we obtain  $\int td\hat\mu_p=+\infty$ 
from \eqref{infty-eq1} and \eqref{infty-eq20}.
In case (c) we obtain
 $\int td\hat\mu_p<+\infty$ 
from \eqref{infty-eq10} and \eqref{infty-eq20}.
\end{proof}
By the finiteness of $\int t d\hat\mu_p$ in Lemma~\ref{finite},
the $\sigma$-invariant measure
\[\sum_{n=1}^\infty\sum_{k=0}^{n-1}
\hat\mu_p|_{\{t=n\}}\circ \sigma^{-k}\] can be normalized to a probability, denoted by $\mu_p$.
By Abramov-Kac's formula connecting entropies of $\mu_p$ and $\hat\mu_p$, and integrals of functions against $\mu_p$ and $\hat\mu_p$,
we have \begin{equation}\label{AK}h(\hat\mu_p)+\int\Phi d\hat\mu_p=F(\mu_p)\int t d\hat\mu_p.\end{equation}
By \cite[Theorem~2.2.9]{MauUrb03} together with the finiteness of $\int\Phi d\hat\mu_p$ in Lemma~\ref{finite},
 $\hat\mu_p$ is the unique equilibrium state for the potential 
$\Phi$, namely
\begin{equation}\label{AKP}P(\Phi)=h(\hat\mu_p)+\int\Phi d\hat\mu_p.\end{equation}
From $P(\Phi)=0$ in Lemma~\ref{zero-p}, \eqref{AK} and \eqref{AKP} we obtain $F(\mu_p)=0$.
Since $P(\phi)=0$ by Lemma~\ref{exist-equi},
 $\mu_p$ is an equilibrium state for the potential $\phi$.
 
 \begin{lemma}\label{unique-equi}
In case (a) or (c) of Proposition~\ref{ex2-prop}, $\mu_p$ is an equilibrium state for the potential $\phi$. 
In case (a), it is the unique equilibrium state for the potential $\phi$. 
 \end{lemma}
 \begin{proof} Let
 $\nu\in\mathcal M(\Sigma,\sigma)$ be an ergodic equilibrium state with $\nu(\hat\Sigma)>0$.
  The normalized restriction of $\nu$ to $\hat\Sigma$, denoted by $\hat\nu$, belongs
 to $\mathcal M(\hat\Sigma,\hat {\sigma}|_{\hat\Sigma})$.
 From $P(\phi)=0$,
 Abramov-Kac's formula and
 $P(\Phi)=0$,
 $\hat\nu$
 is an equilibrium state for the potential $\Phi$,
 namely $\hat\mu_p=\hat\nu$, and so $\mu_p=\nu$.

 
 In case (a),
 let $\Sigma_0=\{1,2\}^\mathbb N\subset\Sigma$.
 Let $\mathcal M(\Sigma_0,\sigma)$
 denote the set of elements of $\mathcal M(\Sigma,\sigma)$ which are supported on $\Sigma_0$. 
  Since $\bigcup_{n=0}\sigma^{-n}\Sigma_0$ is contained in the complement of $\hat\Sigma$,
 any measure in $\mathcal M(\Sigma,\sigma)$ which does not give positive weight to $\hat\Sigma$ is supported on $\Sigma_0$.
 The variational principle for the  subsystem $(\Sigma_0,\phi|_{\Sigma_0})$ gives
 \[\begin{split}\sup_{\nu\in\mathcal M(\Sigma_0,\sigma)}F(\nu)&= \lim_{n\to\infty}\frac{1}{n}\log\sum_{a_1\cdots a_n\in\{1,2\}^n}\sup_{a_1\cdots a_n}\exp\phi\\
 &\leq\log\left(\sup_{[11]\cup[12]}\exp\phi+\sup_{[21]\cup[22]}\exp\phi\right)\leq
 \log\zeta<0.\end{split}\]
 $\mathcal M(\Sigma_0,\sigma)$ does not contain an equilibrium state for the potential $\phi$.
 \end{proof}

In case (a) or (c) of Proposition~\ref{ex2-prop}, let $m$ denote the normalized restriction of $(m_p\otimes{\rm Leb})\circ\pi$ to $\hat\Sigma$.
By the Markov structure of $\hat\sigma|_{\hat\Sigma}$
and the distortion estimate in Lemma~\ref{holder}, we may apply the standard argument 
(see e.g., \cite[Chapter~V, Section~2]{dMevSt93}) to the sequence
$((1/n)\sum_{k=0}^{n-1}m\circ(\hat\sigma|_{\hat\Sigma})^{-k})_{n=1}^\infty$ to obtain
 a convergent subsequence in the weak* topology.
This limit measure is $\hat\sigma|_{\hat\Sigma}$-invariant, and absolutely continuous with respect to $m$, with a bounded uniformly positive density
\cite[Chapter~V, Theorem~2.2]{dMevSt93}.
Hence it satisfies the Gibbs property \eqref{gibbs}.
By the uniqueness of Gibbs state \cite[Theorem~2.2.4]{MauUrb03}, this limit measure is 
$\hat\mu_p$.

It follows that
$\mu_p\circ\pi^{-1}$
is absolutely continuous with respect to $m_p\otimes{\rm Leb},$ and
by \cite[Corollary~3.1]{Mor88}, the density is independent of $\omega.$
The measure $\lambda_p=(\mu_p\circ\pi^{-1})\circ\Pi^{-1}$ on $X$ is a stationary measure that is absolutely continuous  with respect to ${\rm Leb}$ and satisfies 
$(m_p\otimes\lambda_p)\circ\pi=\mu_p.$
This together with Lemma~\ref{unique-equi}
completes the proof of Proposition~\ref{ex2-prop}(a)(c).


In case (b) of Proposition~\ref{ex2-prop},
suppose
 $\nu\in\mathcal M(\Sigma,\sigma)\setminus\{(m_p\otimes\delta_0)\circ\pi\}$ is an ergodic equilibrium state for the potential $\phi$. Then $\nu(\hat\Sigma)>0$.
  The normalized restriction of $\nu$ to $\hat\Sigma$, denoted by $\hat\nu$, belongs
 to $\mathcal M(\hat\Sigma,\hat {\sigma}|_{\hat\Sigma})$
 and satisfies $\int td\hat\nu<\infty$.
 From $P(\phi)=0$,
 Abramov-Kac's formula and
 $P(\Phi)=0$,
 $\hat\nu$
 is an equilibrium state for the potential $\Phi$.
 By \cite[Theorem~2.2.9]{MauUrb03} and \cite[Corollary~2.7.5]{MauUrb03},
 $\hat\nu$ is a Gibbs state, and so $\hat\mu_p=\hat\nu$. This yields
 a contradiction
 to Lemma~\ref{finite}, completing
 the proof of Proposition~\ref{ex2-prop}(b).
 \end{proof}

\subsection{Almost-sure weighted equidistribution along subsequences}\label{conv-sub-sec}
In the case the uniqueness of equilibrium state does not hold
as in Proposition~\ref{ex2-prop}(c), passing to convergent subsequences we maintain a weighted equidistribution of random cycles in the following sense.
\begin{prop}\label{conv-sub}
Let $T_1,\ldots,T_N$ be non-uniformly expanding Markov maps on $X$ generating a nice, topologically mixing skew product Markov map.
  Let $p$ be an $N$-dimensional positive probability vector for which there exist an integer $\ell\geq2$ and stationary measures $\lambda_{p,1},\ldots,\lambda_{p,\ell}$ such that
  ergodic equilibrium states for the random geometric potential $\phi=\phi_p$ are precisely $(m_p\otimes\lambda_{p,i})\circ\pi$, $i=1,\ldots,\ell$.
Then, for $m_p$-almost every sample $\omega\in\Omega$,
any accumulation point of the sequence
$(\xi_{n}^\omega)_{n=1}^\infty$ in the weak* topology is a convex combination of $\lambda_{p,1},\ldots,\lambda_{p,\ell}$.
\end{prop}

\begin{proof}
 Let $\mathcal K=\{\mu\in\mathcal M(\Lambda)\colon I_\Lambda(\mu)=0\}$.
 By the lower semicontinuity of $I_\Lambda$, $\mathcal K$ is a closed subset of $\mathcal M(\Lambda)$.
 The assumption of Proposition~\ref{conv-sub} implies
 \[\mathcal K=\{m_p\otimes(\rho_1\lambda_{p,1}+\cdots+\rho_\ell\lambda_{p,\ell})\colon\rho_1,\ldots,\rho_\ell\in[0,1],\rho_1+\cdots+\rho_\ell=1\}.\]
Note that $\Pi_*(\mathcal K)$ is the set of convex combinations of 
$\lambda_{p,1},\ldots,\lambda_{p,\ell}$.
   Let $\mathcal M(\mathcal K)$ 
denote the set of elements of
$\mathcal M(\mathcal M(\Lambda))$ whose supports are contained in $\mathcal K.$ 
Define a projection
 $\Gamma\colon \tilde\mu\in\mathcal M(\mathcal M(\Lambda))\to \Gamma(\tilde\mu)\in\mathcal M(\Lambda)$ by
 \[\int_{\mathcal M(\Lambda)}
 \left(\int \varphi d\mu\right) d\tilde\mu(\mu)=\int \varphi d\Gamma(\tilde\mu)\quad\text{for any continuous }\varphi\colon\Lambda\to\mathbb R.\]
The left-hand side is a normalized non-negative linear functional on the space of continuous real-valued functions on $\Lambda$, and so
$\Gamma$ is well-defined by Riesz's representation theorem.
It is clear that $\Gamma$ is continuous.
For $\mu\in\mathcal M(\Lambda)$, let $\delta_\mu\in\mathcal M(\mathcal M(\Lambda))$ denote the unit point mass at $\mu$.
For $\mu,\nu\in\mathcal M(\Lambda)$ and $\rho\in[0,1]$ we have
 $\Gamma((1-\rho)\delta_\mu+\rho\delta_\nu)=(1-\rho)\mu+\rho\nu$, which
 implies $\Pi_*\circ\Gamma(\tilde\mu^\omega_n)=\xi^\omega_n$.

Let $\omega\in\Omega$, and let $(\xi_{n_j}^\omega)_{j=1}^\infty$ be a convergent subsequence of $(\xi_{n}^\omega)_{n=1}^\infty$. Taking a further subsequence if necessary we may assume
 $(\tilde{\mu}_{n_j}^\omega)_{j=1}^\infty$ converges to the limit measure $\tilde\mu^\omega$. 
 The continuity of $\Gamma$ shows
  $\Gamma(\tilde\mu^\omega_{n_j})\to\Gamma(\tilde\mu^\omega)$ in the weak* topology as $j\to\infty$.
The argument in the proof of Proposition~\ref{abstract} shows that  $\tilde\mu^\omega\in\mathcal M(\mathcal K)$ holds almost surely.
 As in \cite[Lemma~2.7]{T},   $\Gamma(\mathcal M(\mathcal K))\subset\mathcal K$ holds and therefore
 $\Gamma(\tilde\mu^\omega)\in\mathcal K$. We obtain
 $\xi^\omega_{n_j}=\Pi_*\circ\Gamma(\tilde\mu^\omega_{n_j})\to \Pi_*\circ\Gamma(\tilde\mu^\omega)\in\Pi_*(\mathcal K)$ in the weak* topology as $j\to\infty$,
 which completes the proof.
\end{proof}

\subsection{L-S-V maps}\label{@}
Liverani, Saussol and Vaienti \cite{LSV99} introduced a one-parameter family  $L_\alpha\colon[0,1]\to[0,1]$ $(\alpha>0)$ of maps given by
\[L_\alpha(x)=\begin{cases}x(1+2^\alpha x^\alpha)& x\in\left[0,\frac{1}{2}\right),\\
2x-1&x\in\left[\frac{1}{2},1\right],\end{cases}\]
now called the L-S-V maps after them.
This map has $0$ as a common neutral fixed point.
If $\alpha<1$, 
$L_\alpha$ has negative Schwarzian derivative and
Lebesgue almost every orbit is asymptotically distributed with respect to an invariant probability measure that is absolutely continuous with respect
to the Lebesgue measure.
If $\alpha\geq1$, 
Lebesgue almost every orbit is asymptotically distributed with respect to the unit point mass at $0$.
An interaction of these two compelling behaviors in random setup has attracted attention of researchers.  Statistical properties of random compositions of L-S-V maps with parameters chosen from a fixed compact interval 
according to a fixed distribution were investigated in \cite{BBD14,BBD16,BBR19,BQT21,Gou07}.

Let us consider an i.i.d. random dynamical system generated by finitely many L-S-V maps
$L_{\alpha_1},L_{\alpha_2},\ldots,L_{\alpha_N}$ with
$\alpha_1<\alpha_2<\cdots<\alpha_N$.
The unit point mass
 at $0$ is a stationary measure, and the corresponding measure on the shift space with $2N$ symbols is an equilibrium state for the random geometric potential $\phi$. If $\omega_N<1$, as in Proposition~\ref{ex2-prop}(c) there is another equilibrium state for $\phi$ that corresponds to the stationary measure absolutely continuous with respect to the Lebesgue measure. In particular, (A4) does not hold.  
As in Proposition~\ref{conv-sub}, any accumulation point of the sequence $(\xi_n^\omega)_{n=1}^\infty$ 
 is a convex combination of these two stationary measures, almost surely. 
  In the case $\omega_1\geq1$, one can verify 
  a version of Proposition~\ref{ex2-prop}(b)
  using the distortion technique in \cite[Corollary~3.3]{BQT21}.

\subsection{On extensions of the main results}\label{extension}
Although we have suppressed the setup of our main results to the minimal complexity,
some further extensions can be considered in view of recent advances
in the field of random dynamical systems.
First of all,
it is relevant to weaken the i.i.d. setting to a weakly dependent random noise: to weaken the independence of the driving process
$\theta\colon\Omega\to\Omega$ to a mixing condition satisfied for example by suitable stationary/non-stationary 
Markov chains other than Bernoulli
\cite{ANV15, BoDo99, Kif94}.
Also relevant is to consider
extensions to random dynamical systems expanding on average having
a contracting part \cite{ANV15,Pel84}.
Considering random dynamical systems generated by 
uncountably many maps is completely relevant from the viewpoint of structural stability and bifurcation theory.

Our arguments and results can be easily generalized
to treat distributions of {\it random preimages}. Under the assumption in Theorem~A, fix a point $x_0\in{\rm int}(X)$, and for each $\omega\in\Omega$ consider a Borel probability measure on $\mathcal M(X)$ given by 
 \[\frac{1}{Z_{\omega,n}'}\sum_{x\in{\rm Pre}(T_\omega^n,x_0)} |(T^n_\omega)'x|^{-1}\delta_{\delta_x^{\omega,n}}\quad(n=1,2,\ldots),\]
 where 
  ${\rm Pre}(T_\omega^n,x_0)=\{x\in X\colon T_\omega^n(x)=x_0\}$ and
  $Z_{\omega,n}'$ denotes the normalizing constant.
Slightly modifying the proof of Theorem~A, one can show that this sequence converges in the weak* topology to
 $\delta_{\lambda_p}$ for $m_p$-almost every $\omega$.

\subsection*{Acknowledgments}
We thank Juho Lepp\"anen and Takehiko Morita for fruitful discussions.
SS was supported by the JSPS KAKENHI 20K14331.
HT was supported by the JSPS KAKENHI 
19K21835 and 20H01811.

\end{document}